\documentclass[11pt, reqno]{amsart}

\usepackage{amsmath}
\usepackage{amssymb}
\usepackage{amsthm}
\usepackage{amscd}
\usepackage{color} 
\usepackage{graphicx}
\usepackage{enumitem}
\usepackage{bm}
\usepackage[all]{xy}
\allowdisplaybreaks 

\newcommand{\R}{\mathbb{R}}
\newcommand{\Q}{\mathbb{Q}}
\newcommand{\Z}{\mathbb{Z}}
\newcommand{\F}{\mathbb{F}}

\newcommand{\GL}{\mathit{GL}}
\newcommand{\SL}{\mathit{SL}}
\newcommand{\PSL}{\mathit{PSL}}

\newcommand{\ab}{\mathrm{ab}}
\newcommand{\orb}{\mathrm{orb}}
\newcommand{\Int}{\operatorname{Int}}

\newcommand{\KT}{\mathrm{KT}}

\newcommand{\rank}{\operatorname{rank}}

\newcommand{\ang}[1]{\langle#1\rangle}

\newtheorem{theorem}{Theorem}[section]
\newtheorem{proposition}[theorem]{Proposition}
\newtheorem{corollary}[theorem]{Corollary}
\newtheorem{lemma}[theorem]{Lemma}
\newtheorem{claim}[theorem]{Claim}

\newtheorem{problem}[theorem]{Problem}

\theoremstyle{definition}
\newtheorem{definition}[theorem]{Definition}

\theoremstyle{remark}
\newtheorem{remark}[theorem]{Remark}
\begin{document}

\title{On the genera of symmetric unions of knots}

\author{Michel Boileau}
\address{Aix-Marseille University, CNRS \\
Centrale Marseille, I2M, 3 Pl. Victor Hugo, 13003 Marseille \\
France}
\email{michel.boileau@univ-amu.fr}

\author{Teruaki Kitano}
\address{Department of Information Systems Science, Faculty of Science and Engineering, Soka University \\
Tangi-cho 1-236, Hachioji, Tokyo 192-8577 \\
Japan}
\email{kitano@soka.ac.jp}

\author{Yuta Nozaki}
\address{
Faculty of Environment and Information Sciences, Yokohama National University \\
79-7 Tokiwadai, Hodogaya-ku, Yokohama, 240-8501 \\
Japan\vspace{-0.6em}}
\address{
SKCM$^2$, Hiroshima University \\
1-3-2 Kagamiyama, Higashi-Hiroshima City, Hiroshima, 739-8511 \\
Japan}
\email{nozaki-yuta-vn@ynu.ac.jp}

%57K10 Knot theory
%57K30 General topology of 3-manifolds
%57K32 Hyperbolic 3-manifolds
%57R30 Foliations in differential topology; geometric theory
%57Q10 Simple homotopy type, Whitehead torsion, Reidemeister-Franz torsion, etc.
%57K14 Knot polynomials
%57M05 Fundamental group, presentations, free differential calculus
%57R18 Topology and geometry of orbifolds
%57M12 Low-dimensional topology of special (e.g., branched) coverings
%57M50 General geometric structures on low-dimensional manifolds


\maketitle

\begin{abstract}
In the study of ribbon knots, Lamm introduced symmetric unions inspired by earlier work of Kinoshita and Terasaka.
We show an identity between the twisted Alexander polynomials of a symmetric union and its partial knot.
As a corollary, we obtain an inequality concerning their genera.
It is known that there exists an epimorphism between their knot groups, and thus our inequality provides a positive answer to an old problem of Jonathan Simon in this case.
Our formula also offers a useful condition to constrain possible symmetric union presentations of a given ribbon knot.
It is an open question whether every ribbon knot is a symmetric union.
\end{abstract}

\setcounter{tocdepth}{1}
\tableofcontents

%%%%%%%%%%%%%%%%%%%%%%%%%%
\section{Introduction}
\label{sec:intro}
Let $K$ be a knot in the $3$-sphere $S^3$ and let $E(K)$ denote the exterior $S^3\setminus \Int N(K)$ of $K$, where $N(K)$ is a tubular neighborhood of $K$.
We write $G(K)$ for the knot group of $K$, that is, the fundamental group $\pi_1(E(K))$.
For two knots $K$ and $K'$, if there exists an epimorphism $G(K)\twoheadrightarrow G(K')$, then it is natural to expect that $K$ is more complicated than $K'$ in some sense.
Let $g(K)$ denote the genus of $K$.
In Kirby's list \cite{Kir97}, Jonathan Simon posed the following long standing problem.

\begin{problem}[{\cite[Problem~1.12(B)]{Kir97}}]
\label{prob:Simon}
If there is an epimorphism $\varphi\colon G(K) \twoheadrightarrow G(K')$, does it imply that $g(K)\ge g(K')$?
\end{problem}

Note that Simon mentioned Problem~\ref{prob:Simon} in \cite[p.~410]{NAMS76} as well.
A positive answer to this problem is known when $\deg\Delta_{K'}(t)=2g(K')$ (e.g., $K'$ is fibered or alternating), where $\Delta_{K'}(t)$ denotes the Alexander polynomial.
Indeed, the existence of epimorphisms implies that $\Delta_{K}(t)$ is divisible by $\Delta_{K'}(t)$ (see \cite[Exercise~9 in Chapter~VII]{CrFo77}).
In \cite[Corollary~6.22]{Gab83I} and \cite[Theorem~8.8]{Gab87III}, Gabai gave an affirmative answer when the epimorphism is induced by a proper map of non-zero degree between the knot exteriors.
Also, Friedl and L\"uck~\cite{FrLu19PLMS} solved Problem~\ref{prob:Simon} when the knot group $G(K')$ is residually locally indicable elementary amenable (e.g., $K'$ is fibered).

Moreover, we observe that when the epimorphism $\varphi$ is meridian-preserving and $K'$ is prime then $g(K)\ge g(K')$ if $\varphi(\lambda_K) \neq 1$ in Corollary~\ref{cor:meridian}.
This suggests that it is hard to analyze the case $\varphi(\lambda_K)=1$, which happens, for instance, when $K$ is the connected sum of $K'$ and its mirror image.
Indeed, we have a map $f\colon E(K'\sharp (-K'^\ast)) \to E(K')$ of degree zero such that $f$ is orientation-preserving on $E(K') \subset E(K'\sharp (-K'^\ast))$ and not on $E(-K'^\ast) \subset E(K'\sharp (-K'^\ast))$.

In this paper, we focuses on the \emph{symmetric union} of a knot with its mirror image introduced by Kinoshita and Terasaka~\cite{KiTe57} and generalized by Lamm~\cite{Lam00} as illustrated in Figure~\ref{fig:symmetric_union}.

\begin{figure}[h]
 \centering \includegraphics[width=\textwidth]{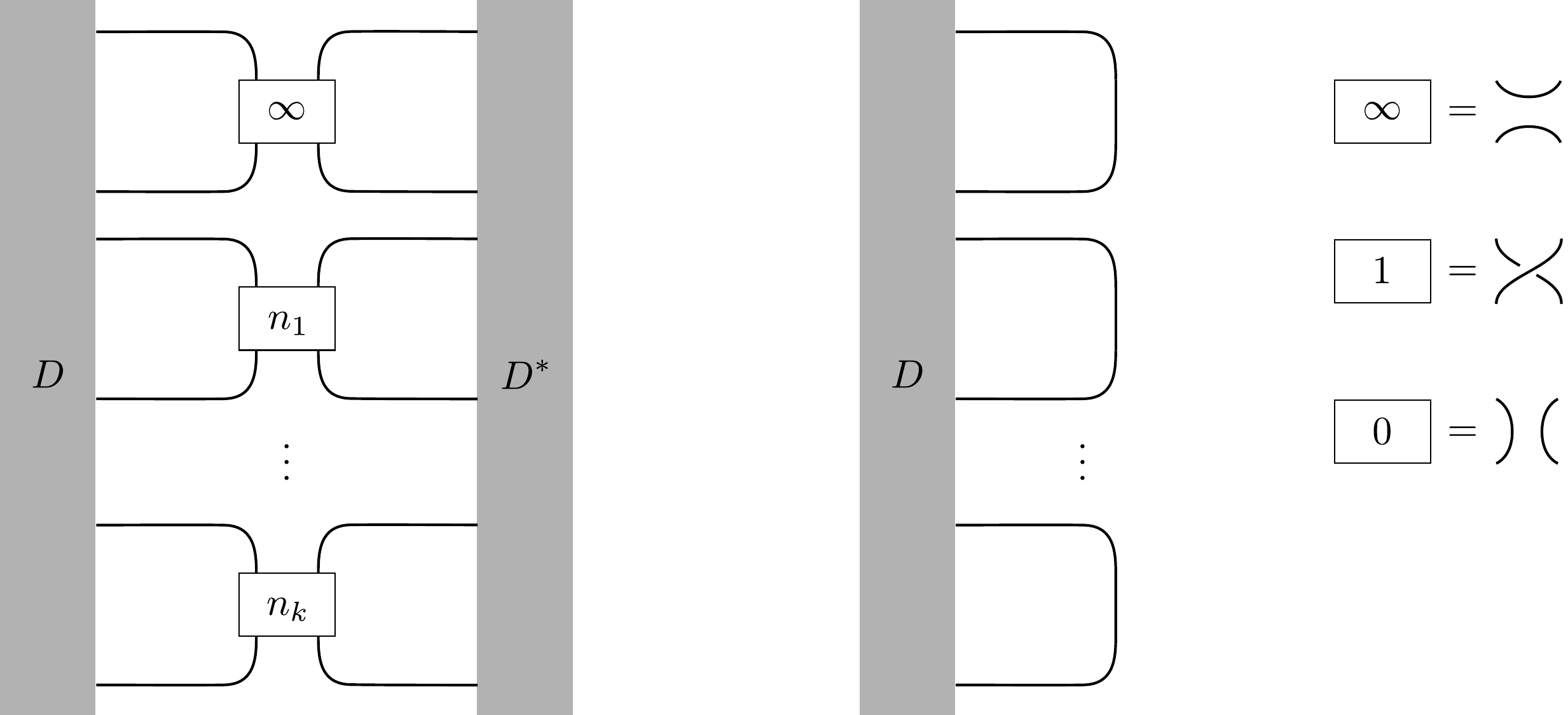}
 \caption{Symmetric union $(D \cup D^\ast)(\infty, n_1, \dots, n_k)$ and its partial knot.}
 \label{fig:symmetric_union}
\end{figure}

\begin{definition}
\label{def:symmetric_union}
Let $D$ be an unoriented planar diagram of a knot $K_D$ and let $D^\ast$ be the diagram obtained from $D$ by reflecting $D$ across an axis in the plane.
Let $B_0, B_1,\dots, B_k$ be balls along the axis, each of which is invariant by the reflection and intersects $D$ in a trivial arc.
One replaces the trivial tangle $(B_0, B_0 \cap (D \cup D^\ast))$ by a $\infty$-tangle to get the connected sum of the diagrams $D$ and $-D^\ast$.
For $i=1,\dots,k$, one replaces each trivial tangle $(B_i, B_i \cap (D\sharp D^\ast))$ by an $n_i$-tangle, where $n_i \in \Z$.
The knot diagram $(D \cup D^\ast)(\infty, n_1, \dots, n_k)$ obtained from $D\cup D^\ast$ in this way is called a \emph{symmetric union} of the diagram $D$ and $D^\ast$.
A knot which admits a diagram $(D \cup D^\ast)(\infty, n_1, \dots, n_k)$ is said to admit a \emph{symmetric union presentation} with \emph{partial knot} $K_D$, where $K_D$ corresponds to the closure of the diagram $D$ such that $(D \cup D^\ast)(0, 0, \dots, 0) = K_D \cup K_{D}^\ast$.
\end{definition}

When there is a single tangle replacement, the construction is due to Kinoshita and Terasaka~\cite{KiTe57}.
The extension to multiple symmetric tangle replacements is due to Lamm~\cite{Lam00}.
The symmetric union construction is not unique and the isotopy type of the knot $K = D \cup D^\ast(\infty, n_1, \dots, n_k)$ depends on both the diagram $D$ and the location of the tangle replacements.
When the partial knot $K_D$ is oriented and all the twist numbers $n_1,\dots,n_k$ are even, the symmetric union $D \cup D^\ast(\infty, n_1, \dots, n_k)$ inherits an orientation from the connected sum $D \sharp (-D^\ast)$, but when some twists $n_i$ are odd, the orientation of $D \cup D^\ast(\infty, n_1, \dots, n_k)$ is not well-defined.
A symmetric union $D \cup D^\ast(\infty, n_1, \dots, n_k)$ is said to be \emph{even} if all the $n_i$ are even, and then we denote it by $D \cup (-D^\ast)(\infty, n_1, \dots, n_k)$.
Otherwise the symmetric union is said to be \emph{skew}.

The even symmetric union construction, which is a generalization of the connected sum of a knot with its mirror image, provides plenty of examples of pairs of knots $K$ and $K'$ with an epimorphism $\varphi\colon G(K) \twoheadrightarrow G(K')$ satisfying $\varphi(\lambda_K)=1$ (see Proposition\ref{prop:epimorphism}).
However, the genus of a symmetric union was not being much studied.
Our result on the twisted Alexander polynomials of even symmetric unions allows to get a lower bound on the genus of an even symmetric union and to give a positive answer to Jonathan Simon's problem in this case.

The symmetric union construction always produces ribbon knots and the following problem is still open.

\begin{problem}[{\cite{Lam00}, \cite{BeFe24}}]
Does every ribbon knot admit a symmetric union presentation?
\end{problem}

Prime ribbon knots up to $10$ crossings and ribbon $2$-bridge knots admit symmetric union presentations due to Lamm~\cite{Lam00,Lam21JKTR}.
This is true also for all but $15$ prime ribbon knots with $11$ and $12$ crossings by Seeliger~\cite{See14} (see \cite{Lam21EM}).
In fact, there is no obstruction known for a ribbon knot to be a symmetric union. 
Moreover there is no a priori upper bound for the number of required twist regions.
However, our result gives some useful conditions to restrain possible even symmetric union presentations of a given knot in Section~\ref{sec:montesinos}.

Despite the dependence on the diagram, a classical fact about symmetric unions is that the Alexander polynomial of it depends only on the parity of the twist numbers as shown in \cite[Theorem 2.4]{Lam00}.
In particular, for an even symmetric union one has the following property.

\begin{proposition}
\label{prop:alexander}
Let $K$ be a knot admitting an even symmetric union presentation with partial knot $K_D$.
Then $\Delta_K(t) = \Delta_{K_D}(t)^{2}$, up to multiplication by a unit in $\Z[t,t^{-1}]$.
\end{proposition}

Another useful property relating an even symmetric union presentation to its partial knot is given by the following proposition, which is due to Michael Eisermann.

\begin{proposition}[{\cite[Theorem~3.3]{Lam00}}]
\label{prop:epimorphism}
Let $K$ be a knot admitting an even symmetric union presentation with partial knot $K_D$.
Then there is an epimorphism $\varphi_D\colon G(K) \twoheadrightarrow G(K_D)$ which sends a meridian of $K$ to a meridian of $K_D$ and which kills the longitude.
\end{proposition}

The epimorphism $\varphi_{D}$ is defined in \cite[Figure~11]{Lam00} by mapping the generators of the Wirtinger presentation of the oriented diagram $D \cup (-D^\ast)(\infty, n_1, \dots, n_k)$ of the even symmetric union $K$ to the corresponding Wirtinger generators of the oriented diagram $D$ of the partial knot $K_D$, according to the Figure~\ref{fig:Lamm_epi}.
The Wirtinger relations are clearly satisfied, and thus this map induces a well-defined epimorphism between the knot groups.

\begin{figure}[h]
 \centering \includegraphics[width=0.9\textwidth]{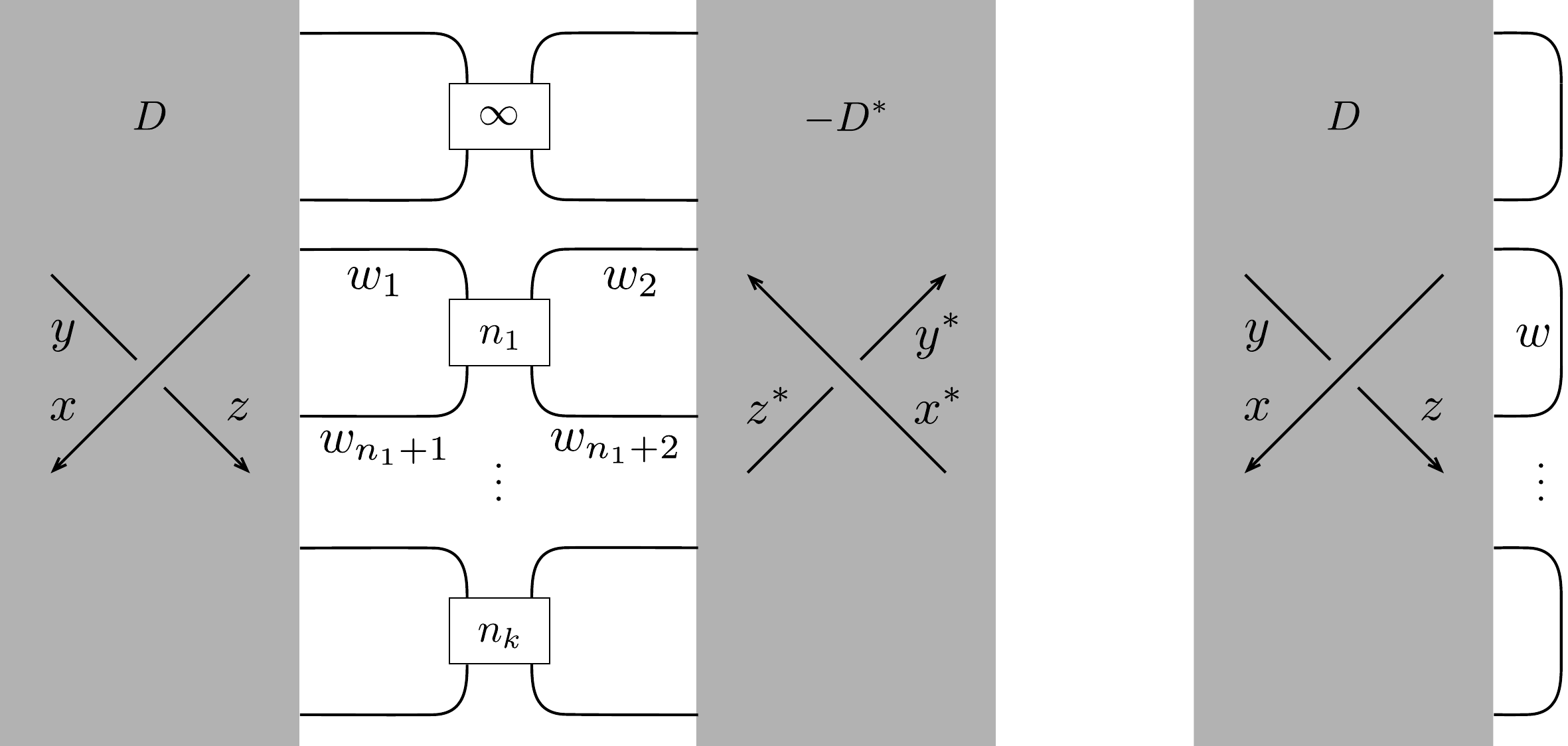}
 \caption{Wirtinger generators $x, y, z, x^\ast, y^\ast, z^\ast, $ satisfy $yx = xz$ and $y^\ast x^\ast = x^\ast z^\ast$, and $\varphi_D(w_i)=w$ when $n_1\geq 0$.}
 \label{fig:Lamm_epi}
\end{figure}

In this article we use the epimorphism $\varphi_D$ given by Proposition~\ref{prop:epimorphism} to generalize Proposition~\ref{prop:alexander} in the setting of twisted Alexander polynomials.
A representation $\rho_D\colon G(K_D) \to \GL(d,\F)$ over a field $\F$, induces a representation $\rho_D\circ\varphi_D\colon G(K)\to \GL(d,\F)$.
We write $\mu_D\in G(K_D)$ for a fixed meridian element and prove the following formula.

\begin{theorem}\label{thm:TAP} 
Let $K$ be a knot admitting an even symmetric union presentation with partial knot $K_D$.
Let $\varphi_D\colon G(K) \twoheadrightarrow G(K_D)$ be the epimorphism given by Proposition~\ref{prop:epimorphism}.
For a representation $\rho_D\colon G(K_D) \to \GL(d,\F)$, the equality
\[
\Delta_{K,\rho_D\circ\varphi_D}(t) = \Delta_{K_D,\rho_D}(t)^2 \det(\rho_D(\mu_D)t-I_d)
\]
holds up to multiplication by a unit in $\F[t,t^{-1}]$.

In particular, $\deg\Delta_{K,\rho_D\circ\varphi_D}(t) = 2\deg\Delta_{K_D,\rho_D}(t) +d$, where $\deg$ is defined to be the degree of a numerator minus that of a denominator.
\end{theorem}

\begin{remark}
While the twisted Alexander polynomial is not symmetric in general, we get $\Delta_{K_D,\rho_D}(t)^2$ since $\Delta_{K_D,\rho_D}(t) = \Delta_{-K_{D}^\ast,\rho_{D}}(t)$ up to multiplication by a unit in $\F[t,t^{-1}]$.
Indeed, the Wirtinger presentation of $-D^\ast$ is completely the same as that of $D$ and we consider the common representation $\rho_{D}$ (see Figure~\ref{fig:Lamm_epi}).
\end{remark}

Combining with \cite{FrNa15} or \cite{FrVi11}, we obtain the following inequality which implies an affirmative answer to Simon's Problem~\ref{prob:Simon} for an even symmetric union. 
%\cite{FrKi06}

\begin{corollary}
\label{cor:genus_ineq}
For a knot $K$ admitting an even symmetric union presentation with partial knot $K_D$, it holds that $g(K)\geq 2g(K_D)$.
\end{corollary}

This corollary extends a result by Moore~\cite{Moo16} to any even symmetric union, where her result is for an even symmetric union with a single twist region and ``symmetric fusion number one'' by using Heegaard-Floer theory.

K\"{o}se~\cite[Theorem~E]{Kos22Thesis} showed that the ribbon Montesinos knot $11a_{201}=K(1/3,2/3,4/5)$ does not admit a symmetric union presentation with a single twist region.
The next result imposes some strong restrictions on any possible symmetric union presentation of this knot.

\begin{proposition}
\label{prop:11a_201}
If the ribbon Montesinos knot 
$11a_{201} =  K(\frac{1}{3}, \frac{2}{3}, \frac{4}{5})$ admits a symmetric union presentation, it must be a skew one with partial knot $6_1$ or $9_1$.
\end{proposition}

It is worth mentioning that the (original) Alexander polynomial does not rule out the possibility of $11a_{201}$ being an even symmetric union of $6_1$ since $\Delta_{11a_{201}}(t)=(2-5t+2t^2)^2=\Delta_{6_1}(t)^2$.

In a forthcoming paper~\cite{BKN2}, we study the relationship between the knot type of a symmetric union and the knot type of the associated partial knot.
In particular, we study whether only finitely many distinct knots can occur as partial knots for a symmetric union presentation of a given knot.

\subsection*{Acknowledgments}
The authors would like to thank Masaaki Suzuki for his help with the computation in Corollary~\ref{cor:11a_201}.
They also thank Jae Choon Cha, Christoph Lamm, and Hidetoshi Masai for valuable comments and discussion.
The authors are grateful to the anonymous referee for helpful comments.
This study was supported in part by JSPS KAKENHI Grant Numbers JP19K03505, JP20K14317, and 23K12974.
The first and second authors were supported by Soka University International Collaborative Research Grant.

%%%%%%%%%%%
\section{Twisted Alexander polynomials}
\label{subsec:TAP}
%%%%%%%%%%%%%%%%%%%%%%%%%%%%%%%%%%
Here we give a quick review of twisted Alexander polynomials. 
Let $K$ be an oriented knot in $S^3$.
By taking a regular diagram $D$ with $N$ crossings, a Wirtinger presentation of $G(K)$ is given as 
\[
G(K)=\ang{x_1,\ldots,x_N \mid r_1,\ldots,r_N},
\]
where each generator corresponds to an arc in the diagram $D$ and each relator comes from a crossing in $D$.
Here it is known that any one relator, for example $r_N$, can be removed and hence:
\[
G(K)=\ang{x_1,\ldots,x_N \mid r_1,\ldots,r_{N-1}}.
\]
The abelianization homomorphism
\[
\alpha\colon G(K)\to H_1(E(K);\Z)\cong\Z = \ang{t}
\]
is given by assigning to each generator $x_i$ the meridian element $t \in H_1(E(K);\Z)$.

Here we consider a representation 
\[
\rho\colon G(K)\to \GL(d,\F),
\]
where $\F$ is a field. 
The maps $\rho$ and $\alpha$ naturally induce two ring homomorphisms 
\[
\tilde{\rho}\colon {\Z}[G(K)] \to M(d,{\F})
\]
and
\[
\tilde{\alpha}\colon \Z[G(K)]\to \Z[\ang{t}]=\Z[t,t^{-1}], 
\]
where ${\Z}[\Gamma]$ denotes the integral group ring of a group $\Gamma$ and $M(d,{\F})$ is the algebra of $m\times m$ matrices over ${\F}$. 

Then $\tilde{\rho}\otimes\tilde{\alpha}$ defines a ring homomorphism ${\Z}[G(K)]\to M(d,{\F}[t,t^{-1}])$.
Let $F_N$ denote the free group on generators $x_1,\ldots,x_N$ and let
\[
\Phi\colon {\Z}[F_N]\to M(d,{\F}[t,t^{-1}])
\]
be the composition of the surjection $\tilde{\phi}\colon {\Z}[F_N]\to{\Z}[G(K)]$ induced by the presentation of $G(K)$ and the map $\tilde{\rho}\otimes\tilde{\alpha}\colon {\Z}[G(K)]\to M(d,{\F}[t,t^{-1}])$.

Let us consider the $(N-1)\times N$ matrix $A$ whose $(i,j)$-entry is the $d\times d$ matrix
\[
\Phi\left(\frac{\partial r_i}{\partial x_j}\right) \in M\left(d,{\F}[t,t^{-1}]\right),
\]
where $\frac{\partial}{\partial x_j}\colon \Z [F_N]\to\Z [F_N]$ ($j=1,\dots,N$) denotes Fox's free differential.
We call $A$ the \emph{twisted Alexander matrix} of $G(K)$ associated to $\rho$. 
For $1\leq j\leq N$, let us denote by $A_{x_j}$ the $(N-1)\times(N-1)$ matrix obtained from $A$ by removing the $j$th column
$\left(
\Phi(\frac{\partial r_1}{\partial x_j}) \dots \Phi(\frac{\partial r_{N-1}}{\partial x_j})
\right)^T$
corresponding to $x_j$.
We regard $A_{x_j}$ as a $d(N-1)\times d(N-1)$ matrix with coefficients in ${\F}[t,t^{-1}]$.

Then Wada's \emph{twisted Alexander polynomial}~\cite{Wad94} of a knot $K$ associated to a representation $\rho\colon G(K)\to \GL(d,\F)$ is defined to be the rational expression
\[
\Delta_{K,\rho}(t) = \frac{\det A_{x_j}}{\det\Phi(x_j-1)}.
\]
This is well-defined up to multiplication by a unit in $\F[t,t^{-1}]$. 
In particular, it does not depend on the choice of a presentation of $G(K)$ under this equivalence.

\begin{remark}
\noindent
The denominator $\det\Phi(x_j-1)$ can be expressed as $\det(\rho(x_j)\alpha(x_j)-I_d)=\det(\rho(\mu)t-I_d)$.
\end{remark}

%%%%%%%%%%%%%%%%%%%%%%%%%%
\section{Proof of Theorem~\ref{thm:TAP}}
\label{sec:proof_of_TAP}

This section is devoted to proving our main theorem.
A key of the proof is to choose a suitable Wirtinger presentation of a given even symmetric union.
In this section, we use $\doteq$ for the equality up to multiplication by a unit in $\F[t,t^{-1}]$.

\begin{figure}[h]
 \centering
 \includegraphics[width=0.95\textwidth]{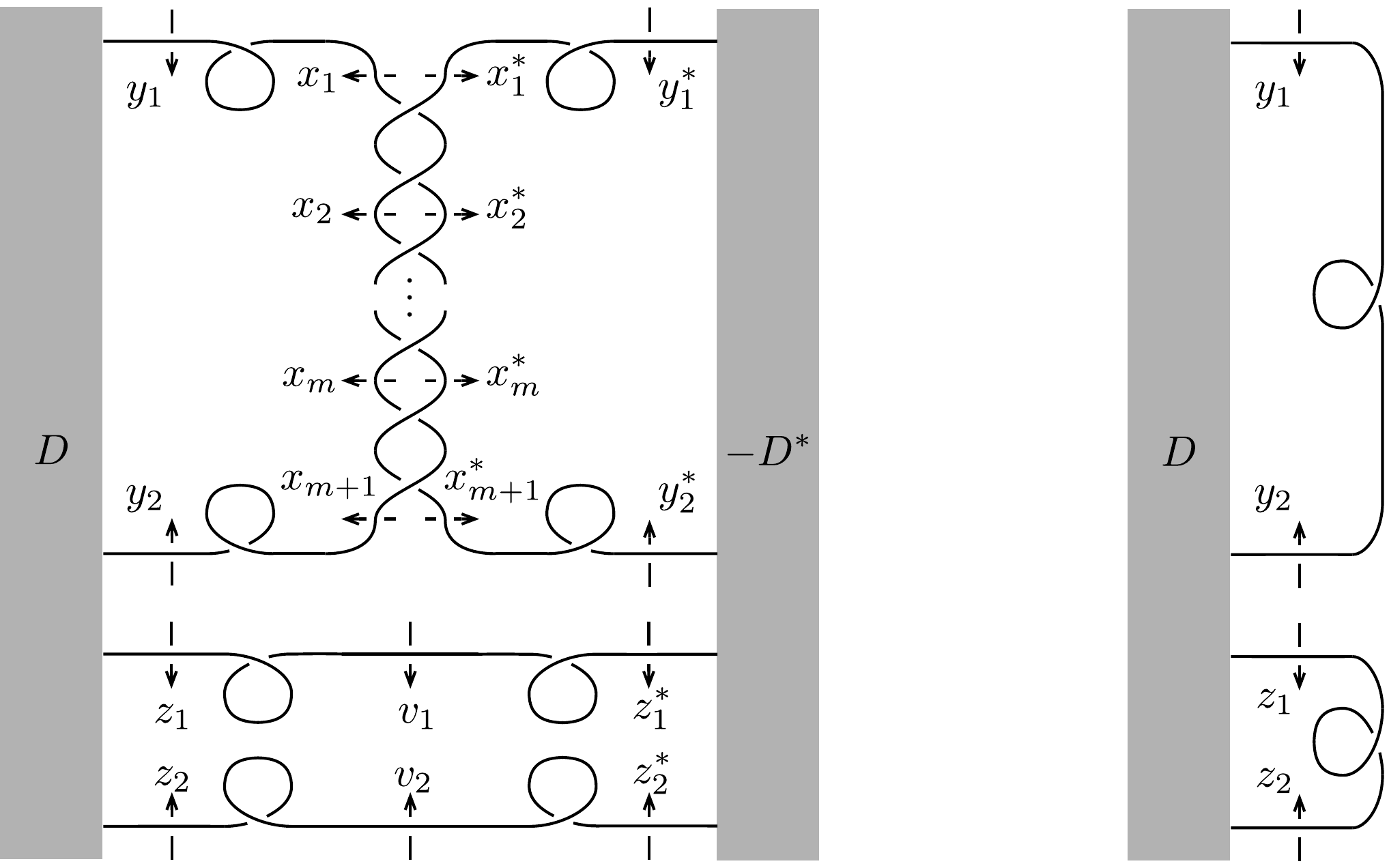}
 \caption{Generators for the knot groups of a symmetric union and a partial knot when $k=1$.}
 \label{fig:wirtinger_pres_1}
\end{figure}

\begin{proof}[Proof of Theorem~\ref{thm:TAP}]
We first show the case $K=D\cup (-D^\ast)(\infty,2m)$, where $m$ is non-negative.
Using the Reidemeister moves, we may assume $K$ is drawn as Figure~\ref{fig:wirtinger_pres_1} and there are $c\ (\geq 3)$ crossings and $c-2$ arcs in $D$.
Then we have a Wirtinger presentation such that generators are
\[
\left\{
\begin{aligned}
& u_1, \dots, u_{c-2}, u_1^\ast, \dots, u_{c-2}^\ast, \\
& v_1, v_2, x_1, x_1^\ast, \dots, x_{m+1}, x_{m+1}^\ast, \\
& y_1, y_1^\ast, y_2, y_2^\ast, z_1, z_1^\ast, z_2, z_2^\ast
\end{aligned}
\right.
\]
and relators are
\[
\left\{
\begin{aligned}
& r_1, \dots, r_c, r_1^\ast, \dots, r_c^\ast, \\
& x_1x_1^\ast x_2^{-1}(x_1^\ast)^{-1}, x_1^\ast x_2^{-1}(x_2^\ast)^{-1}x_2, \dots, x_m x_m^\ast x_{m+1}^{-1}(x_m^\ast)^{-1}, x_m^\ast x_{m+1}^{-1}(x_{m+1}^\ast)^{-1}x_{m+1}, \\
& x_1y_1^{-1}, x_1^\ast(y_1^\ast)^{-1}, x_{m+1}y_2^{-1}, x_{m+1}^\ast(y_2^\ast)^{-1}, v_1z_1^{-1}, v_1(z_1^\ast)^{-1}, v_2z_2^{-1}, v_2(z_2^\ast)^{-1}.
\end{aligned}
\right.
\]
Let us drop the last relator.
Then the corresponding Alexander matrix $A$ is
\begin{gather*}
\left(
\begin{array}{cc|cc|ccccccccc}
U & & & & & & & & & & & & \\
 & U & & & & & & & & & & & \\\hline
 & & & & I & X-I & -X & & & & & & \\
 & & & & & I & X^{-1}-I & -X^{-1} & & & & & \\
 & & & & & & \ddots & & & & \ddots & & \\
 & & & & & & & & & I & X-I & -X & \\
 & & & & & & & & & & I & X^{-1}-I & -X^{-1} \\\hline
 & & & & I & & & & & & & & \\
 & & & & & I & & & & & & & \\
 & & & & & & & & & & & I & \\
 & & & & & & & & & & & & I \\\hline
 & & I & & & & & & & & & & \\
 & & I & & & & & & & & & & \\
 & & & I & & & & & & & & & 
\end{array}
\right.
\hspace{3em}
\\
\hspace{0.5\textwidth}
\left.
\begin{array}{cccc|cccc}
 Y_1 & & Y_2 & & Z_1 & & Z_2 \\
 & Y_1 & & Y_2 & & Z_1 & & Z_2 \\\hline
 & & & & & & & \\
 & & & & & & & \\
 & & & & & & & \\
 & & & & & & & \\
 & & & & & & & \\\hline
 -I & & & & & & & \\
 & -I & & & & & & \\
 & & -I & & & & & \\
 & & & -I & & & & \\\hline
 & & & & -I & & & \\
 & & & & & -I & & \\
 & & & & & &-I & 
\end{array}
\right),
\end{gather*}
where $U=\Phi\left(\frac{\partial r_i}{\partial u_j}\right)_{\substack{1\leq i\leq c\\ 1\leq j\leq c-2}}$, $X=\Phi(x_1)$, $Y_l=\Phi\left(\frac{\partial r_i}{\partial y_l}\right)_{1\leq i\leq c}$, $Z_l=\Phi\left(\frac{\partial r_i}{\partial z_l}\right)_{1\leq i\leq c}$ for $l=1,2$.
Note here that, by the definition of $\varphi_D$ in \cite[Theorem~3.3]{Lam00}, we have $\varphi_D(u_i) = \varphi_D(u_i^\ast) = u_i$ for $i=1,\dots,c-2$, $\varphi_D(x_i) = \varphi_D(x_i^\ast) = \varphi_D(y_j) = \varphi_D(y_j^\ast) = y_j$ for $i=1, \dots, m+1$ and $j=1,2$.
Therefore, $U=\Phi\left(\frac{\partial r_i^\ast}{\partial u_j^\ast}\right)_{\substack{1\leq i\leq c\\ 1\leq j\leq c-2}}$, $X=\Phi(x_i)=\Phi(x_i^\ast)=\Phi(y_j)=\Phi(y_j^\ast)$, $Y_l=\Phi\left(\frac{\partial r_i^\ast}{\partial y_l^\ast}\right)_{1\leq i\leq c}$, and $Z_l=\Phi\left(\frac{\partial r_i^\ast}{\partial z_l^\ast}\right)_{1\leq i\leq c}$.
To compute the numerator of the twisted Alexander polynomial, we remove the last column corresponding to $z_2^\ast$ and compute the determinant of the resulting square matrix $A_{z_2^\ast}$.
Using elementary column operations, we eliminate the seven $I$'s above in the last seven rows (at the level of blocks). 
We get an upper triangular block matrix with a lower $7 \times 7$ diagonal block of $-I$ entries.
Then the determinant is equal to that of the matrix
\[\hspace{-2em}
\left(
\begin{array}{cc|cc|ccccccccc}
U & & Z_1 & Z_2 & Y_1 & & & & & & & Y_2 & \\
 & U & Z_1 & & & Y_1 & & & & & & & Y_2 \\\hline
 & & & & I & X-I & -X & & & & & & \\
 & & & & & I & X^{-1}-I & -X^{-1} & & & & & \\
 & & & & & & \ddots & & & & \ddots & & \\
 & & & & & & & & & I & X-I & -X & \\
 & & & & & & & & & & I & X^{-1}-I & -X^{-1} 
\end{array}
\right)
\]
up to sign.
By adding the $j$th column to the $(j+1)$st column (at the level of blocks) for $j=5,6,\dots,2m+5$ in this order, the rightmost part is transformed into
\[
\left(
\begin{array}{ccccccccc}
Y_1 & Y_1 & Y_1 & & \cdots & & Y_1 & Y_1+Y_2 & Y_1+Y_2 \\
 & Y_1 & Y_1 & & \cdots & & Y_1 & Y_1 & Y_1+Y_2 \\\hline
I & X & O & & & & & & \\
 & I & X^{-1} & O & & & & & \\
 & & \ddots & & & & \ddots & & \\
 & & & & & I & X & O & \\
 & & & & & & I & X^{-1} & O 
\end{array}
\right).
\]
For the resulting entire matrix, we add the second row multiplied by $-1$ to the first row and then add the first column to the second column.
The resulting matrix is
\[
\left(
\begin{array}{cc|cc|ccccccccc}
U & & & Z_2 & Y_1 & O & O & & \cdots & & O & Y_2 & \\
 & U & Z_1 & & & Y_1 & Y_1 & & \cdots & & Y_1 & Y_1 & Y_1+Y_2 \\\hline
 & & & & I & X & O & & & & & & \\
 & & & & & I & X^{-1} & O & & & & \\
 & & & & & & \ddots & & & & \ddots & & \\
 & & & & & & & & & I & X & O & \\
 & & & & & & & & & & I & X^{-1} & O 
\end{array}
\right).
\]
Since $(U\ Z_1\ Y_1+Y_2)$ is a $c\times c$ matrix and the rows of the three columns containing $U$, $Z_1$, $Y_1+Y_2$ are zero except the second row, after successively permuting the last column with each of the previous $2m+2$ columns, the determinant of the above matrix is the product of $\det(U\ Z_1\ Y_1+Y_2)$ and that of
\[
\left(
\begin{array}{c|c|cccccccc}
U & Z_2 & Y_1 & O & & & \cdots & & O & Y_2 \\\hline
 & & I & X & O & & & & & \\
 & & & I & X^{-1} & O & & & \\
 & & & & \ddots & & & & \ddots & \\
 & & & & & & & I & X & O \\
 & & & & & & & & I & X^{-1} 
\end{array}
\right)
\]
up to sign.
Adding the $(2j-1)$st column multiplied by $-X$ to the $2j$th column and adding the $2j$th column multiplied by $-X^{-1}$ to the $(2j+1)$st column for $j=2,3,\dots,m+1$, we transform the above matrix into
\[
\left(
\begin{array}{c|c|cccccccc}
U & Z_2 & Y_1 & -Y_1X & Y_1 & & \cdots & & -Y_1X & Y_1+Y_2 \\\hline
 & & I & O & O & & & & & \\
 & & & I & O & O & & & \\
 & & & & \ddots & & & & \ddots & \\
 & & & & & & & I & O & O \\
 & & & & & & & & I & O 
\end{array}
\right)
\]
whose determinant is $\det(U\ Z_2\ Y_1+Y_2)\det(I^{2m})$ up to sign, after successively permuting the last column with each of the previous $2m$ columns.
Then,
\begin{align}
\det A_{z_2^\ast} \doteq \det(U\ Z_1\ Y_1+Y_2)\det(U\ Z_2\ Y_1+Y_2).
\label{eq:product}
\end{align}

Now, in the case $k=1$ and $m\geq 0$, the desired equality
\[
\Delta_{K,\rho}(t)\det(\rho(\mu)t-I_d) \doteq \left(\Delta_{K_D,\rho_D}(t)\det(\rho_D(\mu_D)t-I_d)\right)^2
\]
follows from \eqref{eq:product} and Claim~\ref{claim:squaredet} whose proof is written after the current proof.

\begin{claim}
\label{claim:squaredet}
For the above matrices $U$, $Y_1$, $Y_2$, $Z_1$, and $Z_2$,
\[
\det(U\ Z_1\ Y_1+Y_2) \doteq \det(U\ Z_2\ Y_1+Y_2) \doteq \Delta_{K_D,\rho_D}(t)\det(\rho_D(\mu_D)t-I_d).
\]
\end{claim}

If $m$ is negative, we use a Wirtinger presentation such that generators are
\[
\left\{
\begin{aligned}
& u_1, \dots, u_{c-2}, u_1^\ast, \dots, u_{c-2}^\ast, \\
& v_1, v_2, x_1^\ast, x_1, \dots, x_{m+1}^\ast, x_{m+1}, \\
& y_1, y_1^\ast, y_2, y_2^\ast, z_1, z_1^\ast, z_2, z_2^\ast
\end{aligned}
\right.
\]
and relators are
\[
\left\{
\begin{aligned}
& r_1, \dots, r_c, r_1^\ast, \dots, r_c^\ast, \\
& x_1^\ast x_1(x_2^\ast)^{-1}x_1^{-1}, x_1(x_2^\ast)^{-1}x_2^{-1}x_2^\ast, \dots, x_m^\ast x_m (x_{m+1}^\ast)^{-1}x_m^{-1}, x_m(x_{m+1}^\ast)^{-1}x_{m+1}^{-1}x_{m+1}^\ast, \\
& x_1y_1^{-1}, x_1^\ast(y_1^\ast)^{-1}, x_{m+1}y_2^{-1}, x_{m+1}^\ast(y_2^\ast)^{-1}, v_1z_1^{-1}, v_1(z_1^\ast)^{-1}, v_2z_2^{-1}, v_2(z_2^\ast)^{-1}.
\end{aligned}
\right.
\]
Note that the order of generators $x_i$ and $x_i^\ast$ is changed.
Then, the left half of the matrix $A$ above is changed to
\[
\left(
\begin{array}{cc|cc|ccccccccc}
U & & & & & & & & & & & & \\
 & U & & & & & & & & & & & \\\hline
 & & & & I & X-I & -X & & & & & & \\
 & & & & & I & X^{-1}-I & -X^{-1} & & & & & \\
 & & & & & & \ddots & & & & \ddots & & \\
 & & & & & & & & & I & X-I & -X & \\
 & & & & & & & & & & I & X^{-1}-I & -X^{-1} \\\hline
 & & & & & I & & & & & & & \\
 & & & & I & & & & & & & & \\
 & & & & & & & & & & & & I \\
 & & & & & & & & & & & I & \\\hline
 & & I & & & & & & & & & & \\
 & & I & & & & & & & & & & \\
 & & & I & & & & & & & & & 
\end{array}
\right.
\]
and the same elementary row and column operations as in the case $m\geq 0$ give the desired equality.

\begin{figure}[h]
 \centering
 \includegraphics[width=0.95\textwidth]{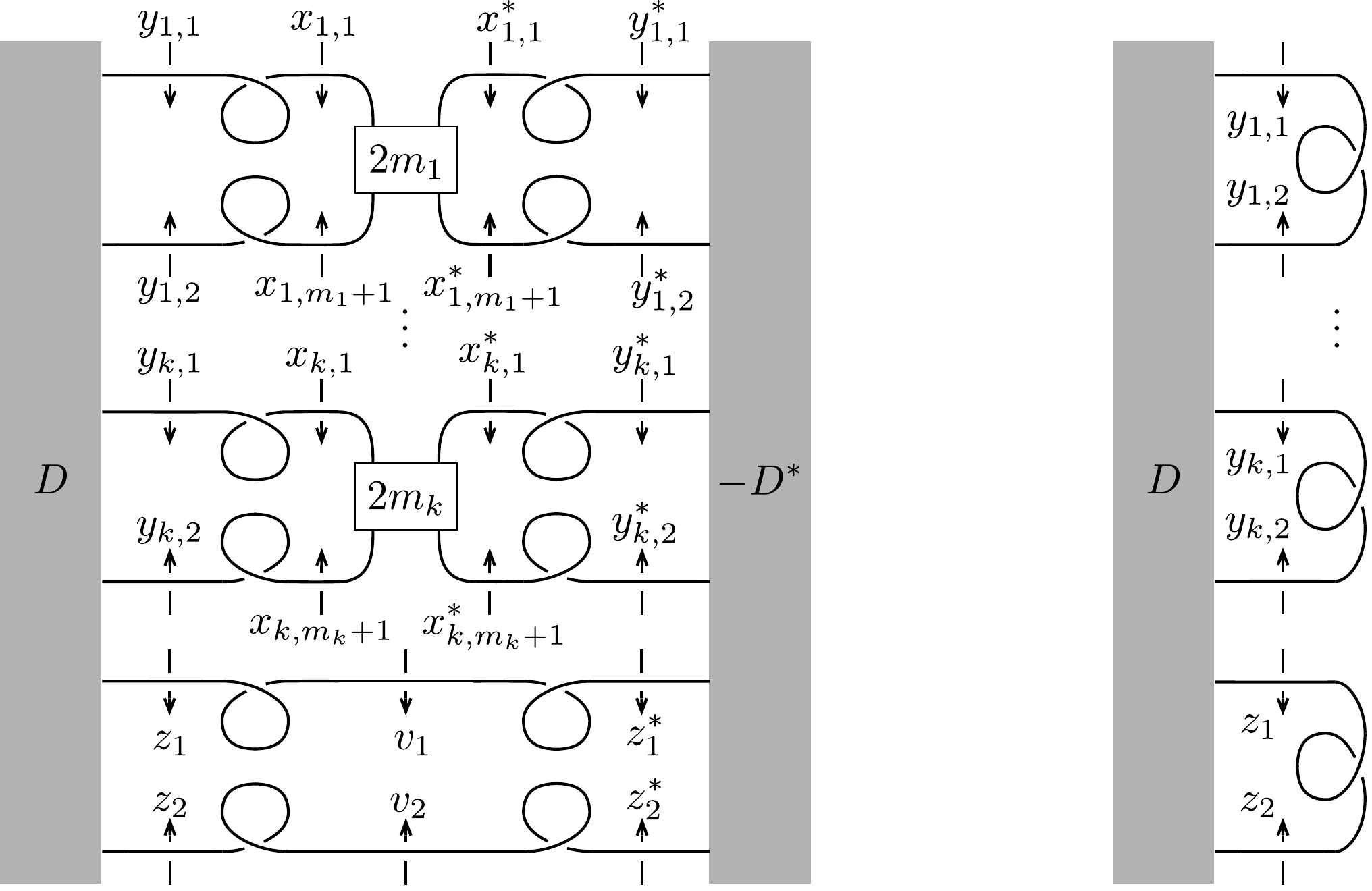}
 \caption{Generators for the knot groups of a symmetric union and a partial knot.}
 \label{fig:wirtinger_pres}
\end{figure}

Next, we consider the case of $K=D\cup (-D^\ast)(\infty,2m_1,\dots,2m_k)$.
For simplicity, we assume $m_l$ are non-negative, but the proof in the negative case is almost the same as mentioned before.
We use a diagram of $K$ as illustrated in Figure~\ref{fig:wirtinger_pres} and a Wirtinger presentation such that generators are
\[
\left\{
\begin{aligned}
& u_1, \dots, u_{c-k-1}, u_1^\ast, \dots, u_{c-k-1}^\ast, \\
& v_1, v_2, x_{1,1}, x_{1,1}^\ast, \dots, x_{1,m_1+1}, x_{1,m_1+1}^\ast, \dots, x_{k,1}, x_{k,1}^\ast, \dots, x_{k,m_k+1}, x_{k,m_k+1}^\ast \\
& y_{1,1}, y_{1,1}^\ast, y_{1,2}, y_{1,2}^\ast, \dots, y_{k,1}, y_{k,1}^\ast, y_{k,2}, y_{k,2}^\ast, z_1, z_1^\ast, z_2, z_2^\ast
\end{aligned}
\right.
\]
and relators are
\[
\hspace{-5em}
\left\{
\begin{aligned}
& r_1, \dots, r_c, r_1^\ast, \dots, r_c^\ast, \\
& x_{1,1}x_{1,1}^\ast x_{1,2}^{-1}(x_{1,1}^\ast)^{-1}, x_{1,1}^\ast x_{1,2}^{-1}(x_{1,2}^\ast)^{-1}x_{1,2}, \dots, x_{1,m_1}x_{1,m_1}^\ast x_{1,m_1+1}^{-1}(x_{1,m_1}^\ast)^{-1}, x_{1,m_1}^\ast x_{1,m_1+1}^{-1}(x_{1,m_1+1}^\ast)^{-1}x_{1,m_1+1}, \\
& \quad\vdots \\
& x_{k,1}x_{k,1}^\ast x_{k,2}^{-1}(x_{k,1}^\ast)^{-1}, x_{k,1}^\ast x_{k,2}^{-1}(x_{k,2}^\ast)^{-1}x_{k,2}, \dots, x_{k,m_k}x_{k,m_k}^\ast x_{k,m_k+1}^{-1}(x_{k,m_k}^\ast)^{-1}, x_{k,m_k}^\ast x_{k,m_k+1}^{-1}(x_{k,m_k+1}^\ast)^{-1}x_{k,m_k+1}, \\
& x_{1,1}y_{1,1}^{-1}, x_{1,1}^\ast(y_{1,1}^\ast)^{-1}, x_{1,m_1+1}y_{1,2}^{-1}, x_{1,m_1+1}^\ast(y_{1,2}^\ast)^{-1},\dots, x_{k,1}y_{k,1}^{-1}, x_{k,1}^\ast(y_{k,1}^\ast)^{-1}, x_{k,m_k+1}y_{k,2}^{-1}, x_{k,m_k+1}^\ast(y_{k,2}^\ast)^{-1}, \\
& v_1z_1^{-1}, v_1(z_1^\ast)^{-1}, v_2z_2^{-1}, v_2(z_2^\ast)^{-1}
\end{aligned}
\right.
\]
Let us drop the same relator $v_2(z_2^\ast)^{-1}$ as in the case $k=1$.
Then the corresponding Alexander matrix is divided into blocks similar to the matrix $A$ above.
After removing the column corresponding to $z_2^\ast$, we apply elementary column operations and obtain the matrix
\[
W=
\left(
\begin{array}{cc|cc|c|c|c}
U & & Z_1 & Z_2 & W_1 & \cdots & W_k \\
 & U & Z_1 & & W_1' & \cdots & W_k' \\\hline
 & & & & W_1'' & & \\\hline
 & & & & & \ddots & \\\hline
 & & & & & & W_k'' 
\end{array}
\right),
\]
where
\begin{gather*}
W_l=
\begin{pmatrix}
 Y_{l,1} & O & \cdots & O & Y_{l,2} & O
\end{pmatrix},
\quad
W_l'=
\begin{pmatrix}
O & Y_{l,1} & O & \cdots & O & Y_{l,2} 
\end{pmatrix},
\\
W_l''=
\begin{pmatrix}
 I & X_l-I & -X_l & & & & & & \\
 & I & X_l^{-1}-I & -X_l^{-1} & & & & & \\
 & & \ddots & & & & \ddots & & \\
 & & & & & I & X_l-I & -X_l & \\
 & & & & & & I & X_l^{-1}-I & -X_l^{-1} 
\end{pmatrix},
\end{gather*}
and $X_l=\Phi(x_{l,i})=\Phi(x_{l,i}^\ast)=\Phi(y_{l,j})=\Phi(y_{l,j}^\ast)$ for $l=1,\dots,k$, $i=1,\dots,m_{l}+1$ and $j=1,2$.
Now, the precess of elementary row and column operations in the case $k=1$ can be applied to the matrix $W$ and we deduce
\[
\det W \doteq \det(U\ Z_1\ Y_{1,1}+Y_{1,2}\ \cdots\ Y_{k,1}+Y_{k,2})\det(U\ Z_2\ Y_{1,1}+Y_{1,2}\ \cdots\ Y_{k,1}+Y_{k,2}).
\]
A computation analogous to the one given in the proof of Claim~\ref{claim:squaredet} shows that 
\begin{align*}
\Delta_{K_D,\rho_D}(t)\det(\rho_D(\mu_D)t-I_d)
&\doteq \det(U\ Z_1\ Y_{1,1}+Y_{1,2}\ \cdots\ Y_{k,1}+Y_{k,2}) \\
&\doteq \det(U\ Z_2\ Y_{1,1}+Y_{1,2}\ \cdots\ Y_{k,1}+Y_{k,2}).
\end{align*}
This completes the proof.
\end{proof}

\begin{proof}[Proof of Claim~\ref{claim:squaredet}]
We focus on a Wirtinger presentation of $K_D$ with generators 
$u_1, \dots, u_{c-2}$, $y_1, y_2, z_1, z_2$
and relators 
$r_1, \dots, r_c$, $y_1y_2^{-1}, z_1z_2^{-1}$.
After dropping the last relation, the corresponding Alexander matrix is
\[
\left(
\begin{array}{c|ccccc}
U & Y_1 & Y_2 & Z_1 & Z_2 \\\hline
 & I & -I & & 
\end{array}
\right).
\]
By removing the last column corresponding to $z_2$ and adding the third column to the second one (at the level of blocks), we obtain
\[
\left(
\begin{array}{c|cccc}
U & Y_1+Y_2 & Y_2 & Z_1 \\\hline
 & & -I & 
\end{array}
\right).
\]
After permuting the last two columns, its determinant is $\det(U\ Y_1+Y_2\ Z_1)$ up to sign.

In the same way, by first removing the fourth column corresponding to $z_1$, the determinant is also equal to $\det(U\ Y_1+Y_2\ Z_2)$ up to sign.
\end{proof}

\begin{proof}[Proof of Corollary~\ref{cor:genus_ineq}]
By \cite[Theorem~1.2]{FrVi15} or \cite[Theorem~1.3]{FrNa15}, there exists a representation $\rho\colon G(K_D)\to \GL(d,\F)$ such that $\deg\Delta_{K,\rho}(t)=d(2g(K_D)-1)$.
Then, we have:
\[
d(2g(K)-1) \geq \deg\Delta_{K,\rho\circ\varphi_D}(t) = 2\deg\Delta_{K_D,\rho}(t)+d = 2d(2g(K_D)-1)+d.
\]
It follows that $2g(K)-1\geq 4g(K_D)-2+1$, and thus $g(K)\geq 2g(K_D)$.
\end{proof}

\begin{remark}
The gap $g(K)-2g(K_D)$ can be arbitrarily large.
In fact, the knot $\KT_{r,n}$ in \cite[Figure~5.10]{Gab86} admits an even symmetric union presentation whose partial knot is the unknot $U$ (see \cite[Figure~15(right)]{Lam00}), and we have $g(\KT_{r,n})-2g(U)=r$ by \cite[Theorem~5.7]{Gab86}.
Moreover, the family of knots $\KT_{r,n}$ implies that there is no upper bound for $g(K)$ in terms of $g(K_D)$ and the integers $k, n_1,\dots,n_k$.
\end{remark}

\begin{remark}
We cannot expect a simple relationship between the twisted Alexander polynomial of a symmetric union which is not necessarily even and that of the associated partial knot (see Section~\ref{sec:skew}).
\end{remark}

The next result follows from Proposition~\ref{prop:alexander} together with Corollary~\ref{cor:genus_ineq}.
It puts a big restriction on the genus of a partial knot for an even symmetric union presentation of a given alternating or fibered knot.

\begin{corollary}
\label{cor:alternating}
Let $K$ be a knot admitting an even symmetric union presentation with partial knot $K_D$.
If $\deg\Delta_K(t)=2g(K)$, then $g(K)=2g(K_D)$.
In particular if $K$ is homologically fibered, then so is $K_D$.
\end{corollary}

Here, a knot $K$ is said to be \emph{homologically fibered} if $\Delta_K(t)$ is monic and $\deg\Delta_K(t)=2g(K)$ (see \cite[Definition~3.1]{GoSa13}).
Recall from \cite[Proposition~5.2]{KiSu08GTM} that if there is an epimorphism $G(K)\to G(K')$ and $K$ is fibered, then $K'$ is also fibered.

\begin{proof}[Proof of Corollary~\ref{cor:alternating}]
A special case of \cite[Theorem~2.4]{Lam00} implies $\Delta_K(t) \doteq \Delta_{K_D}(t)^2$.
Thus $g(K)=\deg\Delta_{K_D}(t)$ by the assumption.
Combining this equality with Theorem~\ref{thm:TAP}, we have
\[
2g(K_D) \geq \deg\Delta_{K_D}(t) = g(K) \geq 2g(K_D).
\]
Hence we conclude that $g(K)=2g(K_D)$.
\end{proof}

\begin{remark}
If a knot $K$ admits an even symmetric union presentation, it follows from Proposition~\ref{prop:alexander} that the degree of its Alexander polynomial is divisible by $4$.
For example the knot $6_1$ is represented as a skew-symmetric union (see \cite{Lam00}), but cannot be represented as any even symmetric union because its Alexander polynomial 
$\Delta_{6_1}(t)=2t^{-1}-5+2t=(2t^{-1}-1)(2t-1)$ has degree $2$.
Furthermore if $K$ admits an even symmetric union presentation and $\deg\Delta_K(t)=2g(K)$, it already follows from the Alexander polynomial that $g(K)$ is even.
\end{remark}

%%%%%%%%%%%%
\section{Skew-symmetric unions}
\label{sec:skew}

This section shows that the inequality obtained in Corollary~\ref{cor:genus_ineq} drastically fails for a skew-symmetric union.

\begin{proposition}
\label{prop:skew}
The genus of a partial knot associated to a skew-symmetric union presentation of a knot $K$ can be arbitrarily larger than the genus of $K$.
\end{proposition}

\begin{proof}
Let $n$ be a positive odd integer and let $K$ be a knot defined as a skew-symmetric union presentation illustrated in Figure~\ref{fig:SkewSymUnion}, where the associated partial knot is the torus knot $T(2,2n+1)$.
On the other hand, $K$ is isotopic to the connected sum of the twist knots $J(2,-n)^\ast=J(2,n-1)$ and $J(2,-n)$ (see Figure~\ref{fig:TwistKnot} for the notation $J(2,n)$).
Here, the facts $g(T(2,2n+1))=n$ and $g(J(2,-n)^\ast \sharp J(2,-n))=2$ complete the proof.
\end{proof}

\begin{figure}[h]
 \centering
 \includegraphics[width=0.99\textwidth]{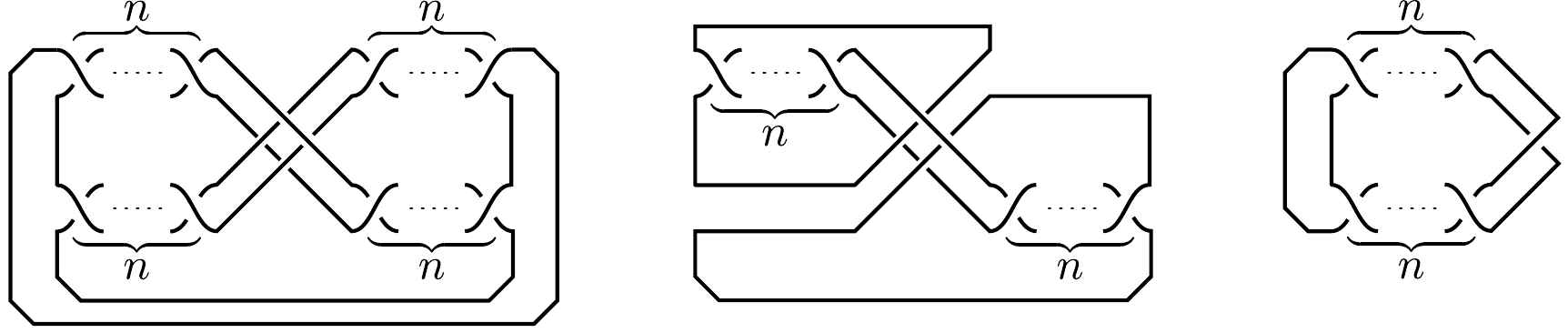}
 \caption{A skew-symmetric union presentation which is isotopic to a connected sum, and the associated partial knot.}
 \label{fig:SkewSymUnion}
\end{figure}

\begin{figure}[h]
 \centering
 \includegraphics[width=0.4\textwidth]{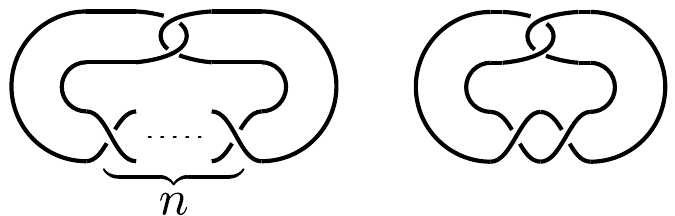}
 \caption{The twist knot $J(2,n)$ for $n\in \Z$ and $J(2,-2)=4_1$.}
 \label{fig:TwistKnot}
\end{figure}

It is worth mentioning here that skew-symmetric unions in Proposition~\ref{prop:skew} can be taken to be hyperbolic.
Indeed, for positive odd integer $n$, let us consider the skew-symmetric union $K_n$ on the left in Figure~\ref{fig:hyperbolic} whose partial knot is again $T(2,2n+1)$, and the knot admits a Seifert surface of genus $2$.
Let $L$ be the $3$-component link $L14n_{47220}^\ast$ drawn on the right in Figure~\ref{fig:hyperbolic}, which is hyperbolic according to SnapPy.
Then the knot obtained from $L$ by Dehn surgery on $U_1$ and $U_2$ with slopes $1/n$ and $-1/n$, respectively, is $K_n$.
By Thurston's Dehn surgery theorem (see \cite{BoPo01} for instance), $K_n$ are hyperbolic for $n$ large enough.

%\cite{BePe92}
%\cite{PePo00}

\begin{figure}[h]
 \centering
 \includegraphics[width=0.99\textwidth]{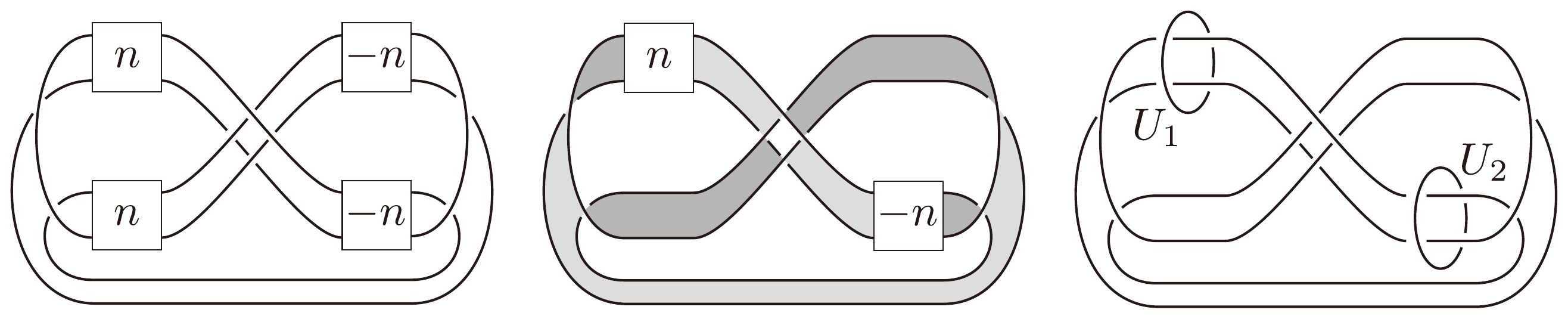}
 \caption{Knot $K_n$, its Seifert surface, and the link $L14n_{47220}^\ast$.}
 \label{fig:hyperbolic}
\end{figure}

%%%%%%%%%%%%%%%%%%%%%%%%%%
\section{The Montesinos knot $11a_{201}$}
\label{sec:montesinos}
In this section, we prove Proposition~\ref{prop:11a_201}. 
We begin by recalling some notations about Montesinos knots and their Seifert fibered $2$-fold branched coverings. 
Let $K = K(\frac{\beta_1}{\alpha_1}, \dots, \frac{\beta_r}{\alpha_r})$ be a Montesinos knot with $r \geq 3$ rational tangles of slopes $\beta_i/\alpha_i \in \Q$ with $\alpha_i > 1$ and $\beta_i \neq 0$ coprime with $\alpha_i$. 
It has been shown by Montesinos~\cite{Mon73} (see also \cite[Chapter~12.D]{BZH14}) that the $2$-fold branched cover of $K$ is the closed orientable Seifert fibered $3$-manifold $\Sigma_2(K) = V(0; e_0; \frac{\beta_1}{\alpha_1}, \dots, \frac{\beta_r}{\alpha_r})$ with base $S^2(\alpha_1, \dots, \alpha_r)$ a $2$-dimensional orientable orbifold with underlying space $S^2$ and $r$ singular points with branching indices $\alpha_i$ corresponding to the $r$ exceptional fibers of types $(\alpha_i, \beta_i)$, where $\alpha_i \geq 2$. 
Its rational Euler number $e_0 = \sum_{i=1}^{r} \frac{\beta_i}{\alpha_i} \in \Q$ verifies that $|\Delta_{K}(-1)| = |H_1(\Sigma_2(K); \Z)| = |e_0| \prod_{i=1}^{r} \alpha_i$ (see \cite[Corollary~6.2]{JaNe83}). 
The Seifert fibered manifold $\Sigma_2(K) = V(0; e_0; \frac{\beta_1}{\alpha_1}, \dots, \frac{\beta_r}{\alpha_r})$ is determined, up to orientation-preserving homeomorphism, by the rational Euler number $e_0 \in \Q$ and the set of fractions $\{\frac{\beta_1}{\alpha_1}, \dots, \frac{\beta_r}{\alpha_r}\}$ in $\Q/\Z$ up to permutations (see \cite{Orl72}, \cite[Theorem~1.5]{JaNe83}), while the Montesinos knot $K(\frac{\beta_1}{\alpha_1}, \dots, \frac{\beta_r}{\alpha_r})$ is determined, up to orientation-preserving homeomorphism of $S^3$,
by the rational Euler number $e_0 \in \Q$ and the set of fractions $\{\frac{\beta_1}{\alpha_1}, \dots, \frac{\beta_r}{\alpha_r}\}$ in $\Q/\Z$ up to dihedral permutations (see \cite[Theorem~12.26]{BZH14}).

We first determine the only possible partial knots for a symmetric union presentation of the ribbon Montesinos knot $11a_{201} = K(\frac{1}{3}, \frac{2}{3}, \frac{4}{5})$.

\begin{proposition}\label{prop:11a_201partial}
The only possible partial knots, up to mirror image, for a symmetric union presentation of the Montesinos knot $11a_{201} = K(\frac{1}{3}, \frac{2}{3}, \frac{4}{5})$ are the $(2,9)$-torus knot or the $2$-bridge knot $6_1$.
\end{proposition}

Before starting the proof, we need to recall some definitions and result from \cite{Lam00}.
One can associate to a knot $K \subset S^3$ the \emph{$\pi$-orbifold group} $G^\orb(K) = G(K) /N$, where $N$ is the subgroup of $\pi_1(S^3 \setminus K)$ normally generated by the square of a meridian (see \cite{BoZi89}).
Let $\Sigma_2(K)$ denote the $2$-fold cover of $S^3$ branched along $K$.
Then the following exact sequence holds:
\[
1 \to \pi_1(\Sigma_2(K)) \to G^\orb(K) \to \Z/2\Z \to 1.
\]

\begin{proposition}[{\cite[Theorem~3.3]{Lam00}}]
\label{prop:orbifold}
Let $K$ be a symmetric union with partial knot $K_D$.
Then there is an epimorphism $\varphi_D^\orb\colon G^\orb(K) \twoheadrightarrow G^\orb(K_D)$ which sends the image of a meridian of $K$ to that of $K_D$ and which kills the image of the preferred longitude of $K$.
\end{proposition}

The epimorphism $\varphi_D^\orb\colon G^\orb(K) \twoheadrightarrow G^\orb(K_D)$ sends the image of a meridian of $K$ to the image of a meridian of $K_D$.
In particular, $\varphi_D^\orb$ sends the subgroup $\pi_1(\Sigma_2(K))$ of index $2$ onto the subgroup $\pi_1(\Sigma_2(K_D))$ of index $2$ since it preserves the images of the meridians.
Therefore $\varphi_D^\orb$ induces an epimorphism $\tilde{\varphi}_D\colon \pi_1(\Sigma_2(K)) \twoheadrightarrow \pi_1(\Sigma_2(K_D))$.
Here let us consider the knot $K= K(\frac{1}{3}, \frac{2}{3}, \frac{4}{5})$. 

\begin{lemma}\label{lem:small}
The $2$-fold branched cover $\Sigma_2(K_D)$ is a small $3$-manifold.
\end{lemma}

\begin{proof}
Let assume that $\Sigma_2(K_D)$ is not small and contains an orientable closed incompressible surface $F$.
Then its fundamental group $\pi_1(\Sigma_2(K_D))$ splits along the fundamental group $\pi_{1}(F)$ as an amalgamated free product or a HNN-extension. 
In particular, $\pi_1(\Sigma_2(K_D))$ acts non-trivially, without edge inversions, on the Bass-Serre tree $\mathcal{T}$ associated to this algebraic decomposition.
The epimorphism $\tilde{\varphi}_D\colon \pi_1(\Sigma_2(K)) \twoheadrightarrow \pi_1(\Sigma_2(K_D))$ induces a non-trivial action, without edge inversions, of the group $\pi_1(\Sigma_2(K))$ on the Bass-Serre tree $\mathcal{T}$.
It follows from \cite{CuSh83} that the manifold $\Sigma_2(K)$ splits along some closed orientable incompressible surface.

The $3$-manifold $\Sigma_2(K) = V(0; \frac{9}{5};\frac{1}{3}, \frac{2}{3}, \frac{4}{5})$ is Seifert fibered with base the hyperbolic $2$-dimensional orbifold $ S^2(3, 3, 5)$ with underlying space  $S^2$ and three singular points with branching indices $\{3, 3, 5\}$. 
By Waldhausen~\cite{Wal67}, a closed incompressible surface in $\Sigma_2(K)$ is either a \emph{vertical} torus which is a union of fibers or a \emph{horizontal} surface transverse to the fibers.
In the first case, the projection of the incompressible vertical torus on the base $S^2(3, 3, 5)$ would be an essential simple closed curve, which does not exist on such an orbifold. 
In the second case, since the base of the Seifert fibration is orientable, the horizontal surface would be non-separating in the rational homology sphere $\Sigma_2(K)$ which is impossible.
Therefore $\Sigma_2(K)$ cannot split along some closed incompressible surface and $\Sigma_2(K_D)$ is a small closed orientable $3$-manifold.
\end{proof}

\begin{lemma}
\label{lem:seifertfibered}
The $2$-fold branched cover $\Sigma_2(K_D)$ is a lens space or a Seifert fibered $3$-manifold with three exceptional fibers and a base orbifold $S^2(\alpha_1, \alpha_2, \alpha_3)$ with underlying space $S^2$ and three singular points with branching indices $\{\alpha_1, \alpha_2, \alpha_3\}$. 
Moreover, at most one of $\alpha_1, \alpha_2, \alpha_3$ can be even.
\end{lemma}

\begin{proof}
If $\pi_1(\Sigma_2(K_D))$ is \emph{finite}, then by the orbifold theorem (see \cite{BoPo01}) the $2$-fold branched cover $\Sigma_2(K_D)$ carries an elliptic geometry, that is to say it is either a lens space or an elliptic Seifert fibered $3$-manifold with three exceptional fibers and finite fundamental group. 
In this last case, the base orbifold $S^2(\alpha_1, \alpha_2, \alpha_3)$ has a finite fundamental group and the triple $(\alpha_1, \alpha_2, \alpha_3)$ is one of the platonic triples $(2, 3, 3),\ (2, 3, 4),\ (2, 3, 5)$ or $(2, 2, n)$ for $n \geq 2$.

Let assume now that $\pi_1(\Sigma_2(K_D))$ is \emph{infinite}.
The $3$-manifold $\Sigma_2(K) = V(0; \frac{9}{5};\frac{1}{3}, \frac{2}{3}, \frac{4}{5})$ is Seifert fibered with base the hyperbolic $2$-dimensional orbifold $ S^2(3, 3, 5)$.
Its fundamental group $\pi_1(\Sigma_2(K))$ is infinite and contains a center $Z$ which is infinite cyclic generated by a regular fiber. 
The quotient $\Gamma_K = \pi_1(\Sigma_2(K))/Z$ is the orbifold fundamental group of the base $ S^2(3, 3, 5)$. 
It is the hyperbolic triangle group $T(3, 3, 5)$ which is a discrete subgroup of $\PSL(2, \R)$ generated by rotations of angles $\{\frac{2\pi}{3}, \frac{2\pi}{3}, \frac{2\pi}{5}\}$ around the vertices of a hyperbolic triangle of angles $\{\frac{\pi}{3}, \frac{\pi}{3}, \frac{\pi}{5}\}$ (see \cite{Kna68}). 
Therefore the group $\Gamma_K$ is generated by torsion elements.

If $\pi_1(\Sigma_2(K_D))$ is centerless, then $\tilde{\varphi}_D(Z) =\{1\}$ and $\tilde{\varphi}_D$ induces an epimorphism from the orbifold group $\Gamma_K = \pi_1(\Sigma_2(K))/Z$ onto $\pi_1(\Sigma_2(K_D))$. 
By Lemma~\ref{lem:small}, the 3-manifold $\Sigma_2(K_D)$ is small, hence it is aspherical since $\pi_1(\Sigma_2(K_D))$ is infinite by assumption. 
Therefore $\pi_1(\Sigma_2(K_D))$ is torsion-free.
On the other hand, $\Gamma_K$ is generated by torsion elements.
Thus, the image $\tilde{\varphi}_D(\Gamma_K)$ must be trivial and this is impossible.
Therefore, $\pi_1(\Sigma_2(K_D))$ has a non-trivial center. 
It follows from \cite{CaJu94} and \cite{Gab92} that $\Sigma_2(K_D)$ is Seifert fibered. 
Since $\Sigma_2(K_D)$ is a rational homology sphere, it is a Seifert fibered manifold with base $S^2$ and $r$ exceptional fibers with $r-1 = \rank(\pi_1(\Sigma_2(K_D)))$ by \cite[Theorem~1.1(ii)]{BoZi84}. 
Moreover $\rank(\pi_1(\Sigma_2(K_D))) \leq \rank(\pi_1(\Sigma_2(K))) = 2$.
Since $\Sigma_2(K_D)$ is aspherical, it cannot be a lens space nor $S^1 \times S^2$, and thus $\Sigma_2(K_D)$ must have $\geq 3$ exceptional fibers. 
Therefore $\rank(\pi_1(\Sigma_2(K_D))) \geq 2$, and hence $\rank(\pi_1(\Sigma_2(K_D))) = 2$. 
It follows that $\Sigma_2(K_D)$ is a Seifert fibered manifold with $r = 3$ exceptional fibers and a base with underlying space $S^2$ and three singular points with branching indices $\{\alpha_1, \alpha_2, \alpha_3\}$.

Since $K_D$ is a knot, 
\[
|H_1(\Sigma_2(K_D); \Z)| = |e_0|\alpha_{1}\alpha_{2}\alpha_{3} = |\beta_1\alpha_2 \alpha_3 + \alpha_1\beta_2 \alpha_3 + \alpha_1\alpha_2 \beta_3|
\]
is odd. 
This implies that at most one of the $\alpha_i$ can be even
\end{proof}

\begin{proof}[Proof of Proposition~\ref{prop:11a_201partial}]
If the Montesinos knot $K = K(\frac{1}{3}, \frac{2}{3}, \frac{4}{5})$ is a symmetric union with partial knot $K_D$, then, by Lemma~\ref{lem:seifertfibered}, $\Sigma_2(K_D)$ is a lens space or a Seifert fibered $3$-manifold with three exceptional fibers and a base orbifold $S^2(\alpha_1, \alpha_2, \alpha_3)$.

If $\Sigma_2(K_D)$ is a lens space, then, by the orbifold theorem, $K_D$ is a $2$-bridge knot.
Since $|H_1(\Sigma_2(K_D); \mathbb{Z})| = {\det K_D} = \sqrt{\det K} = 9$, $K_D$ is the torus knot $9_1$ of type $(2,9)$ or the $2$-bridge knot $K(\frac{9}{2}) = 6_1$ up to reversal of orientation.

If $\Sigma_2(K_D)$ is  a Seifert fibered $3$-manifold with three exceptional fibers, the base orbifold $S^2(\alpha_1, \alpha_2, \alpha_3)$ can be elliptic, euclidean or hyperbolic. Since at most one of $\alpha_1, \alpha_2, \alpha_3$ can be even, in the elliptic case 
the triple $(\alpha_1, \alpha_2, \alpha_3)$ can be only one of the platonic triples $(2, 3, 3)$ or $(2, 3, 5)$, while in the euclidean case it can be only $(3, 3, 3)$. 
In the hyperbolic case the orbifold group $\pi_1^\orb(S^2(\alpha_1, \alpha_2, \alpha_3))$ is the hyperbolic triangle group $T(\alpha_1, \alpha_2, \alpha_3)$ with $\frac{1}{\alpha_1} + \frac{1}{\alpha_2} + \frac{1}{\alpha_3} < 1$ which is a discrete subgroup of $\PSL(2, \R)$ (see \cite{Kna68}). It follows that in all the cases the orbifold group $\pi_1^\orb(S^2(\alpha_1, \alpha_2, \alpha_3))$ has no center. 
Let $\pi_D\colon \pi_1(\Sigma_2(K_D)) \twoheadrightarrow \pi_1^\orb(S^2(\alpha_1, \alpha_2, \alpha_3))$ be the quotient epimorphism by the center.
Since $\pi_1^\orb(S^2(\alpha_1, \alpha_2, \alpha_3))$ has no center, the epimorphism
\[
\pi_D \circ \tilde{\varphi}_D\colon \pi_1(\Sigma_2(K)) \twoheadrightarrow \pi_1(\Sigma_2(K_D)) \twoheadrightarrow \pi_1^\orb(S^2(\alpha_1, \alpha_2, \alpha_3))
\]
kills the center of $\pi_1(\Sigma_2(K))$ and induces an epimorphism $\bar{\varphi}_D\colon \pi_1^\orb(S^2(3, 3, 5)) \twoheadrightarrow \pi_1^\orb(S^2(\alpha_1, \alpha_2, \alpha_3))$ between the orbifold groups of the bases. 
Given the presentation of the group
\[
\pi_1^\orb(S^2(3, 3, 5)) = \ang{x, y, z \mid x^3 = y^3 = z^5 =xyz = 1},
\]
the image $\bar{\varphi}_{D}(z)$ is not trivial, otherwise
\[
\pi_1^\orb(S^2(\alpha_1, \alpha_2, \alpha_3)) = \bar{\varphi}_{D}(\pi_1^\orb(S^2(3, 3, 5))) \cong \Z/3\Z,
\]
which is not possible. 
Therefore $\bar{\varphi}_{D}(z)$ must be of order $5$ in $\pi_1^\orb(S^2(\alpha_1, \alpha_2, \alpha_3))$. 
This means that the group $\pi_1^\orb(S^2(\alpha_1, \alpha_2, \alpha_3))$ must contain an element of order $5$. 
This is not possible for the finite group $\pi_1^\orb(S^2(2, 3, 3))$ which is of order $12$ nor the euclidean group $\pi_1^\orb(S^2(3, 3, 3))$ since the elements of finite order in this group have order $3$. 
Therefore either the orbifold group $\pi_1^\orb(S^2(\alpha_1, \alpha_2, \alpha_3))$ is the (finite) icosahedral group $I = T(2, 3, 5)$ or it is a hyperbolic triangle group $T(\alpha_1, \alpha_2, \alpha_3)$.

If $\pi_1^\orb(S^2(\alpha_1, \alpha_2, \alpha_3))$ is the icosahedral group, $\Sigma_2(K_D)$ is an elliptic Seifert fibered $3$-manifold $V(0; e_0; \frac{1}{2}, \frac{\beta_2}{3}, \frac{\beta_3}{5})$, and then $e_0 = \frac{1}{2} +\frac{\beta_2}{3} + \frac{\beta_3}{5}$ and $| H_1(\Sigma_2(K_D); \Z) | = 30\vert e_0 \vert = \vert 15 + 10 \beta_2 + 6\beta_3\vert$. 
Since $\beta_2$ is coprime with $3$, $|H_1(\Sigma_2(K_D); \Z)| =  \det K_D \ = \sqrt{\det K}$ is coprime with $3$ and cannot be equal to $9$. 
Therefore the orbifold group $\pi_1^\orb(S^2(\alpha_1, \alpha_2, \alpha_3))$ is the hyperbolic triangle group $T(\alpha_1, \alpha_2, \alpha_3)$.

Let us consider the induced epimorphism
\[
\bar{\varphi}_D \colon \pi_1^\orb(S^2(3, 3, 5)) = T(3, 3, 5) \twoheadrightarrow \pi_1^\orb(S^2(\alpha_1, \alpha_2, \alpha_3))= T(\alpha_1, \alpha_2, \alpha_3)
\]
between the orbifold fundamental groups of the bases.
The presentation of the triangle group $T(3, 3, 5) =  \langle x, y, z \mid x^3 = y^3 = z^5 =xyz = 1 \rangle$ shows that each image $\bar{\varphi}_{D}(x)$, $\bar{\varphi}_{D}(y)$ and $\bar{\varphi}_{D}(z)$ is not trivial otherwise $T(\alpha_1, \alpha_2, \alpha_3) = \bar{\varphi}_D(T(3, 3, 5))$ would be trivial or $\Z/3\Z$ which is not possible.
Therefore $\bar{\varphi}_{D}(x)$ and $\bar{\varphi}_{D}(y)$ are of order $3$ while $\bar{\varphi}_{D}(z)$ is of order $5$.
Thus, the two elliptic elements $a = \bar{\varphi}_{D}(y)$ and $b = \bar{\varphi}_{D}(z)$ generate the discrete non-abelian group $T(\alpha_1, \alpha_2, \alpha_3) \subset \PSL(2, \R)$. 
Up to taking suitable powers $u=a^{k}$ and $v = b^{\ell}$ to normalize the matrix representatives of $u$ and $v$ in $\SL(2,\R)$, it follows that at least one of Cases (I)--(VII) in \cite[Theorem~2.3]{Kna68} holds true. 
Since the triangle group $T(\alpha_1, \alpha_2, \alpha_3)$ is co-compact, Case (II) is impossible by \cite[Figures~2 and 3]{Kna68} which exhibit non-compact fundamental domains in this case.
Cases (III) and (VI) do not hold because the generators $u$ and $v$ do not have the same order. 
Case (IV) is not possible because none of the generators $u$ or $v$ has order $2$. 
Cases (V) and (VII) are not possible too because the generator $v$ is of order $5 < 7$. 
Therefore, only Case (I) can be true, and thus, by \cite[Proposition~2.2]{Kna68}, $u$ and $v$ generate a triangle group $T(3, 5, n)$ with $n \geq 3$ since the group $T(3, 5, n)$ is infinite. Since we have the epimorphism 
\[
\bar{\varphi}_D \colon \pi_1^\orb(S^2(3, 3, 5)) = T(3, 3, 5) \twoheadrightarrow \pi_1^\orb(S^2(\alpha_1, \alpha_2, \alpha_3))= T(3, 3, n),
\]
\cite[Lemma~2.5]{Ron92} implies $-\chi(S^2(3, 3, 5)) \geq -\chi(S^2(3, 3, n))$.
So, $1 - (\frac{1}{3} + \frac{1}{5} + \frac{1}{3}) \geq 1 - (\frac{1}{3} + \frac{1}{5} + \frac{1}{n})$, and hence $n \leq 3$. 
It follows that $n= 3$.
Therefore, $T(\alpha_1, \alpha_2, \alpha_3) = T(3, 5, 3)$ and the induced epimorphism
\[
\bar{\varphi}_D \colon \pi_1^\orb(S^2(3, 3, 5)) = T(3, 3, 5) \twoheadrightarrow \pi_1^\orb(S^2(\alpha_1, \alpha_2, \alpha_3))= T(3, 3, 5)
\]
must be an isomorphism because triangle groups are hopfian.
In particular, $(\alpha_1, \alpha_2, \alpha_3) = (3, 3, 5)$ up to permutations.
We have the two exact sequences:
\begin{gather*}
 1 \to Z \cong \Z \to \pi_1(\Sigma_2(K)) \to T(3, 3, 5) \to 1, \\
 1 \to \tilde{\varphi}_{D}(Z)\cong \Z \to \pi_1(\Sigma_2(K_D)) \to T(3, 3, 5) \to 1.
\end{gather*}
The epimorphism $\tilde{\varphi}_D\colon \pi_1(\Sigma_2(K)) \twoheadrightarrow \pi_1(\Sigma_2(K_D))$ induces injective homomorphisms  both on the center $Z$ and on the base $T(3, 3, 5)$. 
Hence it induces an injective homomorphism on $\pi_1(\Sigma_2(K))$. 
Therefore $\tilde{\varphi}_D$ is an isomorphism which induces an isomorphism between the first homology groups $H_1(\Sigma_2(K); \Z)$ and $H_1(\Sigma_2(K_D); \Z)$. 
Then the orders of the first homology groups must be the same.
This is not the case since
\[
|H_1(\Sigma_2(K_D); \Z)| = \det K_D = 9 \text{ and } |H_1(\Sigma_2(K); \Z)| = \det K = 81.
\]
It follows that $\Sigma_2(K_D)$ cannot be a Seifert fibered $3$-manifold with three exceptional fibers and hence it can only be a lens space.
Therefore the only possibilities for $K_D$ are the $(2,9)$-torus knot or the $2$-bridge knot $6_1$.
\end{proof}

\begin{corollary}\label{cor:11a_201}
The Montesinos knot $K(\frac{1}{3}, \frac{2}{3}, \frac{4}{5})$ cannot be an even symmetric union.
\end{corollary}

\begin{proof} 
Let assume that the Montesinos knot $K= K(\frac{1}{3}, \frac{2}{3}, \frac{4}{5})$ admits an even symmetric union presentation with partial knot $K_D$.
By Corollary~\ref{cor:genus_ineq}, $2 = g(K)\geq 2g(K_D)$, hence $g(K_D) \leq 1$. 
By Proposition~\ref{prop:11a_201}, $K_D$ can be only the $2$-bridge knot $6_1$ since the torus knot $9_1$ of type $(2,9)$ has genus $4$.

To rule out the possibility of $K$ being an even symmetric union with partial knot $6_1$, one cannot use the genuine Alexander polynomial since $\Delta_{11a_{201}}(t)=(2-5t+2t^2)^2=\Delta_{6_1}(t)^2$.
However this can be done by using twisted Alexander polynomials and Theorem~\ref{thm:TAP}.

Let $G(6_1) = \langle x, y \mid w(x,y) = 1 \rangle$ be the standard one-relator presentation of the group of the $2$-bridge knot $6_1$.
Consider the representation $\rho_0\colon G(6_1) \to \SL(2, \mathbb{F}_7)$ given by $\rho_0(x) =
\begin{pmatrix}
 0 & 1\\ 6 & 4 
\end{pmatrix}$
and $\rho_0(y) =
\begin{pmatrix}
 0 & 2 \\ 3 & 4
\end{pmatrix}$.
Then the associated twisted Alexander polynomial is $\Delta_{6_{1},\rho_0}(t) = 1$ up to multiplication by units in $\F_7[t^{\pm 1}]$.

If $K$ is an even symmetric union with partial knot $6_1$, then, by Theorem~\ref{thm:TAP},
\[
\Delta_{K,\rho_0\circ\varphi_D}(t) = \Delta_{6_1,\rho_0}(t)^2 \det(\rho_0(\mu_D)t-I_2) = 1 + 3t + t^2
\]
up to multiplication by units in $\F_7[t^{\pm 1}]$.
Using Mathematica, we can check that for the $33$ representations $\rho\colon G(K) \to \SL(2, \mathbb{F}_7)$, $\Delta_{K,\rho}(t) \neq 1 + 3t + t^2$ holds up to multiplication by units in $\F_7[t^{\pm 1}]$.
\end{proof}

Now, Proposition~\ref{prop:11a_201} follows from Proposition~\ref{prop:11a_201partial} together with Corollary~\ref{cor:11a_201}.

\begin{remark}
The fact that the torus knot $9_1$ cannot be a partial knot for an even symmetric union presentation of the Montesinos knot $K(\frac{1}{3}, \frac{2}{3}, \frac{4}{5})$ can be deduced also from the fact that the degree of the Alexander polynomial of the Montesinos knot 
$K(\frac{1}{3}, \frac{2}{3}, \frac{4}{5})$  is twice the degree of the Alexander polynomial of a partial knot: $\deg\Delta_{K}(t) = 4 <2 \deg\Delta_{9_1}(t) = 16$.
\end{remark}

\begin{remark}
The property that $\Delta_{K,\rho_0\circ\varphi}(t)$ is divisible by $\Delta_{K_D,\rho_0}(t)$ is not sufficient to get a contradiction since $\Delta_{K,\rho}(t)$ has $1 + 3t + t^2$ as a factor for some representations $\rho$.
\end{remark}

\begin{remark}
In \cite{HKMS11}, $\SL(2,\F_{11})$-representations were already used to show that there is no epimorphism from $G(K)$ onto $G(6_1)$ sending a meridian of $K$ to a meridian of $6_1$, but Theorem~\ref{thm:TAP} allows us to do that by using $\SL(2,\F_{7})$-representations.
\end{remark}

%%%%%%%%%%%%%%%%%%%%%%%%%%
%%%%%%%%%%%%%%%%%%%%%%%%%%
%%%%%%%%%%%%%%%%%%%%%%%%%%
%%%%%%%%%%%%%%%%%%%%%%%%%%
%%%%%%%%%%%%
\appendix
\section{Epimorphisms which do not kill the longitude}
Jonathan Simon's question remains open even when the epimorphism $\varphi\colon G(K) \twoheadrightarrow G(K')$ is not killing the preferred longitude $\lambda_K$ of the knot $K$. 
In this appendix we give some conditions on the target knot $K'$ which are sufficient to get a positive answer when $\varphi(\lambda_K) \neq 1$.

\begin{proposition}
\label{prop:nonvanishing}
Let $K$ and $K'$ be two knots in $S^3$ and $\varphi\colon G(K) \twoheadrightarrow G(K')$ be an epimorphism such that $\varphi(\lambda_K) \neq 1$.
If $K'$ is not a satellite knot with winding number $1$ nor a cable knot, then $g(K) \geq g(K')$.
In particular, it is true if $K'$ is a hyperbolic knot.
\end{proposition}

The following lemma must be known.
We give a proof for completeness.

\begin{lemma}
\label{lem:nonvanishing}
Let $K$ and $K'$ be two knots in $S^3$ and $\varphi\colon G(K) \twoheadrightarrow G(K')$ be an epimorphism.
If $\varphi(\lambda_K) \in \pi_1(\partial E(K')) \setminus\{1\}$, then $g(K) \geq g(K')$.
\end{lemma}

\begin{proof}
We can assume that the two knots are non-trivial since the lemma is obvious when $K$ or $K'$ is a trivial knot.
Since $\varphi(\lambda_K) \in \pi_1(\partial E(K')) \setminus\{1\}$, it follows from \cite[Lemma~6.9]{BRW14} that the epimorphism $\varphi$ can be realized by a proper map $f\colon E(K)\to E(K')$ of non-zero degree.
Then the inequality $g(K) \geq g(K')$ follows from \cite[Corollary~6.22]{Gab83I} and \cite[Theorem~8.8]{Gab87III}.
\end{proof}

\begin{lemma}\label{lem:peripheral}
Let $\varphi\colon G(K) \twoheadrightarrow G(K')$ be an epimorphism.
If $\varphi(\lambda_K) \neq 1$, then $\varphi (\pi_1(\partial E(K))) \cong \Z \oplus \Z$.
\end{lemma}

\begin{proof}
Since $G(K')$ is torsion-free, $\varphi (\pi_1(\partial E(K)))$ is isomorphic to $\Z$ or $\Z \oplus \Z$.
If $\varphi (\pi_1(\partial E(K))) \cong \Z$, there is a non-trivial element $\mu_{K}^p \lambda_{K}^q \in \pi_1(\partial E(K))$ such that $\varphi(\mu_{K}^p \lambda_{K}^q) = 1$.
Taking the abelianization leads to $p[\varphi(\mu_K)] = 0 \in H_1(E(K'); \Z)$.
Since $[\varphi(\mu_K)] $ generates $H_1(E(K'); \Z) \cong \Z$, we have $p=0$ and $q\neq 0$.
Hence $\varphi(\lambda_K)^q = \varphi(\lambda_{K}^q) = 1$.
This contradicts the hypothesis $\varphi(\lambda_K)\neq 1$ since $G(K')$ is torsion-free.
\end{proof}

\begin{proof}[Proof of Proposition~\ref{prop:nonvanishing}]
By Lemma~\ref{lem:peripheral}, $\varphi (\pi_1(\partial E(K))) \cong \Z \oplus \Z$.
Then, by the enclosing property of the JSJ decomposition of $E(K')$, the image $\varphi (\pi_1(\partial E(K)))$ is conjugated to the fundamental group $\pi_1(W)$ of some geometric piece $W$  of the JSJ decomposition of $E(K')$.
So, after conjugation, one can assume that $\varphi (\pi_1(\partial E(K))) \subset \pi_1(W)$.

\begin{figure}[h]
 \centering
 \includegraphics[width=0.4\textwidth]{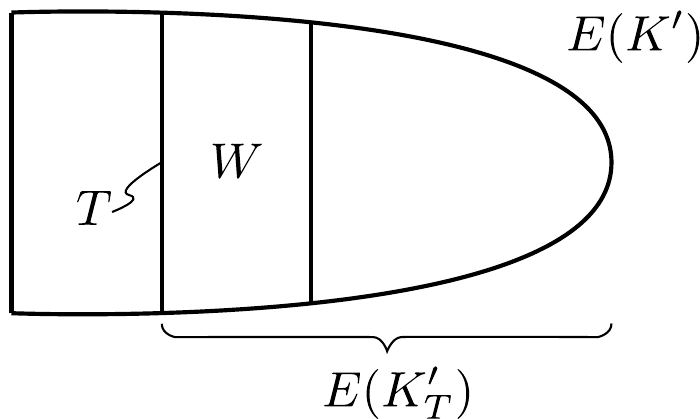}
 \caption{JSJ decomposition of $E(K')$ when $\partial E(K')$ does not belong to $\partial W$.}
 \label{fig:JSJ}
\end{figure}

If $\partial E(K')$ does not belong to $\partial W$, then there exists a JSJ torus $T \subset \partial W \subset S^3$ such that $T$ bounds the exterior $E(K'_T)$ of a knot $K'_T$ and $W \subset E(K'_{T}) \subset E(K')$ (see Figure~\ref{fig:JSJ}).
That is, $K'$ is a satellite of the companion $K'_T$.
The composite map $\varphi (\pi_1(\partial E(K))) \to G(K'_{T})^\ab \to G(K')^\ab$ induced by the inclusions is surjective because $\varphi(\mu_K)$ generates $G(K')^\ab \cong \Z$.
It follows that the morphism $G(K'_{T})^\ab \to G(K')^\ab$ is surjective.
The winding number of $K'$ with respect to its companion $K'_T$ is the index of the image of $G(K'_{T})^\ab$ in $G(K')^\ab$, and thus it is equal to $1$.

So the hypothesis on $K'$ implies that $\partial E(K')$ belongs to $\partial W$.
If $W$ is Seifert fibered, then $W$ is a torus knot exterior, a cable space or a composing space by \cite[Lemma~IX.22]{Jac80}.

If $W$ is a torus knot exterior, $K'$ is a torus knot.
In this case, the inequality $g(K) \geq g(K')$ holds since $\deg\Delta_{K'}(t)=2g(K')$ (see the argument after Problem~\ref{prob:Simon}).
The JSJ piece $W$ cannot be a cable space since $K'$ is not a cable knot.
If $W$ is a composing space, then $K'$ is a composite knot by \cite[Theorem~1]{Sim76}, which contradicts the hypothesis that $K'$ is not a satellite knot of winding number $1$.

Hence $W$ must be a hyperbolic piece.
Since $W$ is atoroidal, $\varphi (\pi_1(\partial E(K))) \cong \Z \oplus \Z$ is conjugated into a peripheral subgroup of $\pi_1(W)$, which corresponds to a torus $T' \subset \partial W$.
If $T' \neq \partial E(K')$, it corresponds to $\partial E(K'_{T'})$ for a companion $K'_{T'}$, and $K'$ is a satellite knot with winding number $1$ since $\varphi (\pi_1(\partial E(K))) \subset \pi_1(\partial E(K'_{T'}))$ generates $G(K')^\ab$.
This contradicts the hypothesis on $K'$.
Hence, we conclude that $T' = \partial E(K')$ and $\varphi(\lambda_K)$ is a non-trivial peripheral element of $G(K')$.
It follows from Lemma~\ref{lem:nonvanishing} that $g(K) \geq g(K')$.
\end{proof}

Another consequence of Lemma~\ref{lem:nonvanishing} concerns the case of meridian-preserving epimorphisms.

\begin{corollary}
\label{cor:meridian}
Let $K$ and $K'$ be two knots in $S^3$ and $\varphi\colon G(K) \twoheadrightarrow G(K')$ be an epimorphism such that $\varphi(\lambda_K) \neq 1$.
If $\varphi(\mu_K) \in \pi_1(\partial E(K'))$ and $K'$ is prime, then $g(K) \geq g(K')$.
\end{corollary}

\begin{proof}
The surjectivity implies that $\varphi (\mu_K) \in \pi_1(\partial E(K'))$ is primitive, and thus it is represented by a simple closed curve on $\partial E(K')$.
Since $\varphi(\mu_K)$ normally generates the knot group $G(K')$, the proof of the property P for knots (\cite{KrMr04}) implies that $\varphi (\mu_K)=\mu_{K'}^{\pm 1}$.
Then, Lemma~\ref{lem:centralizer} below shows that $\varphi(\lambda_K) \in \pi_1(\partial E(K')) \setminus \{1\}$ since it lies in the centralizer of $\varphi (\mu_K)$.
Hence, Lemma~\ref{lem:nonvanishing} completes the proof.
\end{proof}

\begin{lemma}\label{lem:centralizer}
Let $K' \subset S^3$ be a prime knot.
The centralizer $C(\mu_{K'})$ of the meridian element $\mu_{K'}$ in $G(K')$ is the peripheral subgroup $\pi_1(\partial E(K'))$.
\end{lemma}

\begin{proof}
By \cite[Theorem~1]{Sim76}, if there is an element $\gamma \in G(K') \setminus \pi_1(\partial E(K'))$ which commutes with $\mu_{K'}$, then $K'$ must be a composite or a torus knot, or a cable knot.
Since $K'$ is prime, it is a torus knot or a cable knot. 
The centralizer $C(\mu_{K'})$ contains the peripheral subgroup $\pi_1(\partial E(K'))$ and hence it is not cyclic. 
By \cite[Theorem~VI.1.6(i)]{JaSh79} (see also \cite{Fri11}), up to conjugation, $C(\mu_{K'})$ coincides with the centralizer of $\mu_{K'}$ either in the fundamental group of the torus knot exterior $K'$ or in that of the cable piece of $E(K')$ containing $\partial E(K')$.
Since in each case $\mu_{K'}$ is not the center of the fundamental group of the torus knot exterior or of the cable piece, the centralizer $C(\mu_{K'})$ must be abelian by \cite[Addendum~VI.1.8(iia')]{JaSh79}. 
Hence $C(\mu_{K'}) = \pi_1(\partial E(K'))$ because the peripheral subgroup $\pi_1(\partial E(K'))$ is a maximal abelian subgroup of the knot group (see \cite[Corollary~1]{Feu70}).
\end{proof}

Since the exterior of a composite knot admits a degree-one proper map to the exterior of any of its factor, the following is a straightforward consequence of Corollary~\ref{cor:meridian}.

\begin{corollary}\label{cor:factor}
Let $K$ and $K'$ be two knots in $S^3$ and $\varphi\colon G(K) \twoheadrightarrow G(K')$ be an epimorphism such that $\varphi(\lambda_K) \neq 1$.
If $\varphi(\mu_K)$ is conjugated to $\mu_{K'}^{\pm1}$ and $K'$ has a prime factorization $K'_1 \sharp \cdots \sharp K'_n$, then $g(K) \geq \max \{g(K'_1), \dots, g(K'_n) \}$.
\end{corollary}

%%%%%%%%%%%%%%%%%%%%%%%%%
\section{Longitude-killing epimorphisms}
We write $K_1\geq K_2$ if there exists a meridian-preserving epimorphism from $G(K_1)$ onto $G(K_2)$. 
In \cite{KiSu08AM}, it is shown that the following pair can be realized by a degree-zero map, that is, by a longitude-killing epimorphism.

\begin{proposition}
\label{prop:degreezero}
Any pair of the list
\[
8_{10}, 8_{20}, 9_{24}, 10_{62}, 10_{65}, 10_{77}, 10_{82}, 10_{87},10_{99}, 10_{140}, 10_{143}\geq 3_1,\ 
10_{59}, 10_{137} \geq 4_1
\]
can be realized by a degree-zero map.
\end{proposition}

The twisted homology $H_1(E(K);\Q[t,t^{-1}]_\alpha)$ of $E(K)$ by the abelianization $\alpha\colon G(K)\rightarrow \ang{t} \subset \GL(1,\Q[t,t^{-1}])$ is called the \emph{Alexander module} over $\Q$.
It is seen that an epimorphism between knot groups induces an epimorphism between Alexander modules. 
The following proposition is useful to see that an epimorphism can be induced by a non-zero degree map. 
It can be proved by the argument similar to \cite[Lemma~2.2]{Wal99} (see also \cite[Proposition~15]{BBRW16}).

\begin{proposition}
If an epimorphism $\varphi\colon G(K)\rightarrow G(K')$ is induced by a non-zero degree map $(E(K),\partial E(K))\rightarrow 
(E(K'),\partial E(K'))$, then the induced map
\[
\varphi_*\colon H_1(E(K);\Q[t,t^{-1}]_\alpha)\rightarrow H_1(E(K');\Q[t,t^{-1}]_\alpha)
\]
on the Alexander modules over $\Q[t,t^{-1}]$ splits.
Furthermore, if $\varphi$ is induced by a degree-one map, then the induced map
\[
\varphi_*\colon H_1(E(K);\Z[t,t^{-1}]_\alpha)\rightarrow H_1(E(K');\Z[t,t^{-1}]_\alpha)
\]
on the Alexander modules over $\Z[t,t^{-1}]$ splits.
\end{proposition}

By computing the second Alexander polynomial $\Delta_K^{(2)}(t)$, we can see the non-existence of a splitting on Alexander modules.
For example, for the pair $8_{10}\geq 3_1$, since $\Delta_{8_{10}}(t)=\Delta_{3_1}(t)^3=(t^2-t+1)^3$ and $\Delta_{8_{10}}^{(2)}(t)=1$, we have $H_1(E(8_{10});\Q[t,t^{-1}]_\alpha)\cong \Q[t,t^{-1}]/(\Delta_{3_1}(t)^3)$. 
Similarly, $H_1(E(3_{1});\Q[t,t^{-1}]_\alpha)\cong \Q[t,t^{-1}]/(t^2-t+1)$ by $\Delta^{(2)}_{3_1}(t)=1$. 
Hence there is no section for an epimorphism 
\[
H_1(E(8_{10});\Q[t,t^{-1}]_\alpha)\rightarrow H_1(E(3_1);\Q[t,t^{-1}]_\alpha)
\]
and it is induced by only a degree-zero map. 

Similarly, because the second Alexander polynomials of 
\[
8_{10}, 8_{20}, 
9_{24}, 
10_{62}
10_{65}, 
10_{77}, 
10_{82}, 
10_{87}, 
10_{140}, 
10_{143}, 
10_{59}, 10_{137}
\] 
are trivial, 
then these Alexander modules can be determined by the Alexander polynomials only. 
Hence it is easy to see that there exist no sections for the pairs of the above with $3_1$ or $4_1$.
For $10_{99}$, since $\Delta_{10_{99}}(t)=(t^2-t+1)^4$, $\Delta_{10_{99}}^{(2)}(t)=(t^2-t+1)^2$ and $\Delta_{10_{99}}^{(3)}(t)=1$ by the Mathematica package KnotTheory\`{}, it is also seen that
\[
H_1(E(10_{99};\Q[t,t^{-1}]_\alpha))\cong \Q[t,t^{-1}]/((t^2-t+1)^2)\oplus \Q[t,t^{-1}]/((t^2-t+1)^2)
\]
and then there exists no section for $10_{99}\geq 3_1$.
In summary, we have the following. 

\begin{proposition}
Any pair of 
\[
8_{10}, 
9_{24}, 
10_{62}, 
10_{65}, 
10_{77}, 
10_{82}, 
10_{87},
10_{99}, 
10_{143}
\geq 3_1
\quad
10_{59}, 10_{137} \geq 4_1
\]
can be realized by only a degree-zero map. 
\end{proposition}

Although a  symmetric union of even type admits a longitude killing epimorphism, 
the converse is not true in general. 
More precisely, we can prove the following. 
Here realization is proved in \cite{Lam00}. 

\begin{proposition}
Any pair of 
\[
8_{10}, 
9_{24}, 
10_{62}, 
10_{65}, 
10_{77}, 
10_{82}, 
10_{87},
10_{99}, 
10_{143}
\geq 3_1
\quad
10_{59}\geq 4_1
\]
cannot be realized by a symmetric union of even type. 
On the other hand, 
\[
8_{20}, 10_{140}\geq 3_1,\ 10_{137}\geq 4_1
\]
can be realized by a symmetric union of even type. 
\end{proposition}

It is natural to ask whether every longitude-killing epimorphism $\varphi\colon G(K) \to G(K')$ factors through an epimorphism onto $G(K'')$ for some even symmetric union $K''$ with partial knot $K'$. 
However the following example shows that it is not the case.

\begin{proposition}
\label{prop:no_factorization}
No longitude-killing epimorphism $G(9_{24}) \rightarrow G(3_1)$ can factorize through an epimorphism onto $G(K)$ for any even symmetric union $K$ with partial knot $3_1$.
\end{proposition}

Proposition~\ref{prop:no_factorization} follows from the following lemma.

\begin{lemma}
\label{lem:noevenfactor}
There is no epimorphism $\varphi\colon G(9_{24}) \rightarrow G(K)$ for any even symmetric union $K$ with partial knot $3_1$.
\end{lemma}

\begin{proof}
Let assume that there is an epimorphism $\varphi\colon G(9_{24}) \rightarrow G(K)$ for some even symmetric union  $K$ with partial knot $3_1$. Since the knot $9_{24}$ is fibered with genus $3$, by \cite[Proposition~5.2]{KiSu08GTM}, the knot $K$ must be fibered with genus $\leq 3$.
Since $K$ is an even symmetric union with partial knot $3_1$, its Alexander polynomial equals $\Delta_{3_1}(t)^2$ and hence its degree is $4$. 
It follows that $K$ is fibered with genus $2$.

It follows from the proof of \cite[Proposition~5.2]{KiSu08GTM} that the epimorphism $\varphi\colon G(9_{24}) \rightarrow G(K)$ induces an epimorphism from the fundamental group of the fiber surface of the knot $9_{24}$ onto that of the knot $K$.
If $\varphi$ is a longitude-killing epimorphism, then $\varphi$ induces an epimorphism from the fundamental group of a closed orientable surface of genus $3$ onto the free group of rank $4$. 
This is not possible since there is an epimorphism of the fundamental group of a closed, orientable surface of genus $3$ onto a free group of rank $k$ if and only if $k \leq 3$ by \cite[Corollary~3.3]{Jac68}.
Therefore $\varphi (\lambda_0) \neq 1$ for the preferred longitude $\lambda_0$ of the knot $9_{24}$.

The knot $9_{24}$ is the Montesinos knot 
$K(\frac{1}{3}, \frac{2}{3}, \frac{3}{2})$. 
It is a small knot by \cite[Corollary~4(a)]{Oer84}. 
Since the knot $9_{24}$ is small and the epimorphism $\varphi\colon G(9_{24}) \rightarrow G(K)$ is not killing the longitude, the knot $K$ cannot be a satellite knot by \cite[Proposition~1.6]{BBRW10}. 
Therefore $K$ is either a torus knot or a hyperbolic knot. 
Since $K$ is an even symmetric union, it is a ribbon knot and $K$ cannot be a torus knot.

Therefore $K$ is a fibered hyperbolic knot of genus $2$. 
Moreover, $|H_1(\Sigma_2(K); \mathbb{Z})| = \det K = (\det 3_1)^2 = 9$.
If $K$ is a $2$-bridge knot, then $K$ is the $(2,9)$-torus knot or the knot $K(\frac{9}{2}) = 6_1$ up to reversal of orientation. 
This is not possible since the torus knot of type $(2,9)$ has genus $4$ and the knot $6_1$ has genus $1$.
Therefore $K$ must have $\geq 3$ bridges.

Let $\{\mu_0, \lambda_0\}$ be a meridian and preferred longitude pair on the boundary $\partial E(9_{24})$ of the exterior of the knot $9_{24}$.
Since $\varphi(\lambda_0) \neq 1$, $\varphi (\pi_1(\partial E(9_{24}))) \cong \Z \oplus \Z$ by Lemma~\ref{lem:peripheral}.
Then $\varphi (\pi_1(\partial E(9_{24})))$ is conjugate to a subgroup of $\partial E(K)$ since $E(K)$ is a hyperbolic manifold.
Therefore, after conjugation in $\pi_1(E(K))$ one can assume that $\varphi(\mu_0)$ belongs to $\pi_1(\partial E(K))$.
Then the argument of the beginning of the proof of Corollary~\ref{cor:meridian} shows that $\varphi(\mu_0) = \mu_{K}^{\pm1}$ and so the epimorphism $\varphi$ is meridian-preserving.
Therefore it induces an epimorphism $\bar{\varphi}\colon \Sigma_2(9_{24}) \to \Sigma_2(K)$ between the $2$-fold branched covers of the knots $9_{24}$ and $K$.

The $2$-fold branched cover $\Sigma_2(9_{24}) = V(0; \frac{5}{2}; \frac{1}{3}, \frac{2}{3}, \frac{3}{2})$ is a Seifert fibered $3$-manifold with finite fundamental group (see \cite[Section~6.2]{Orl72}).
It follows that $\Sigma_2(K)$ has a finite fundamental group.
By the orbifold theorem, $\Sigma_2(K)$ is a Seifert fibered $3$-manifold $V(0; e_0; \frac{\beta_1}{\alpha_1}, \frac{\beta_2}{\alpha_2}, \frac{\beta_3}{\alpha_3})$ with $(\alpha_1, \alpha_2, \alpha_3) \in \{(2, 3, 3),\,(2, 3, 5)\}$ and $e_0 = \frac{\beta_1}{\alpha_1} + \frac{\beta_2}{\alpha_2} + \frac{\beta_3}{\alpha_3} \in \Q$. 
These are the only platonic triples with at most one of $\alpha_1, \alpha_2, \alpha_3$ even by Lemma~\ref{lem:seifertfibered}.

If $(\alpha_1, \alpha_2, \alpha_3) = (2, 3, 5)$, then $| H_1(\Sigma_2(K_D); \Z) | = 30\vert e_0 \vert = \vert 15 + 10 \beta_2 + 6\beta_3\vert$. 
Since $\beta_2$ is coprime with $3$, so is $|H_1(\Sigma_2(K_D); \Z)|$.
This is not possible since $\vert H_1(\Sigma_2(K); \Z) \vert = \det K = (\det 3_1)^2 = 9$ is not coprime with $3$. 

Therefore $(\alpha_1, \alpha_2, \alpha_3) = (2, 3, 3)$ and $e_0 = \frac{\beta_1}{2} + \frac{\beta_2}{3} + \frac{\beta_3}{3}$. 
Since $|e_0| = \frac{\vert H_1(\Sigma_2(K); \Z) \vert}{\alpha_{1}\alpha_{2}\alpha_{3}} = \frac{9}{18} = \frac{1}{2}$, up to reversal of orientation, the only possibility is $\Sigma_2(K) = V(0; \frac{1}{2}; -\frac{1}{2}, \frac{1}{3}, \frac{2}{3})$ which corresponds to the Montesinos knot $8_{20} = K(-\frac{1}{2}, \frac{1}{3}, \frac{2}{3})$.
In \cite{KiSu05, KiSu08AM}, the partial order induced by meridian-preserving epimorphism between prime knot groups is completely determined for the Rolfsen table of knots up to $10$ crossings. 
It shows that there is no meridian-preserving epimorphism between $G(9_{24})$ and $G(8_{20})$. 
This finishes the proof of the lemma.
\end{proof}

\end{document}